\documentclass[11pt,hyp,]{nyjm}
\usepackage{hyperref}
\hypersetup{nesting=true,debug=true,naturalnames=true}
\usepackage{graphicx,amssymb,upref}
\usepackage{amssymb}
\usepackage{color}
\usepackage{epsfig}
\usepackage{calrsfs}

\let\<\langle
\let\>\rangle

\let\uml\"

\numberwithin{equation}{section}

\newtheorem{Thm}{Theorem}[section]
\newtheorem{Prop}[Thm]{Proposition}
\newtheorem{Lem}[Thm]{Lemma}
\newtheorem{Cor}[Thm]{Corollary}
\theoremstyle{definition}
\newtheorem{Expl}[Thm]{Example}
\newtheorem{Expls}[Thm]{Examples}
\newtheorem{Rem}[Thm]{Remark}
\newtheorem{Def}[Thm]{Definition}

\DeclareMathAlphabet{\pazocal}{OMS}{zplm}{m}{n}

\newcommand{\diam}{\operatorname{diam}}
\newcommand{\spt}{\operatorname{spt}}
\newcommand{\F}{\operatorname{Fill}}
\newcommand{\Lev}{\operatorname{Lev}}
\newcommand{\bM}{\mathbf{M}}
\newcommand{\B}{\mathbf{B}}
\newcommand{\oB}{\mathbf{U}}
\newcommand{\Lip}{\operatorname{Lip}}

\newcommand{\N}{\mathbb{N}}
\newcommand{\R}{\mathbb{R}}
\newcommand{\Z}{\mathbb{Z}}
\newcommand{\cD}{\mathcal{D}}
\newcommand{\cF}{\mathcal{F}}
\newcommand{\cH}{\mathcal{H}}
\newcommand{\cL}{\mathcal{L}}

\newcommand{\cM}{\mathcal{M}}
\newcommand{\cP}{\mathcal{P}}
\newcommand{\cR}{\mathcal{R}}

\newcommand{\defl}{\mathrel{\mathop:}=}

\newcommand{\curr}[1]{[\![{#1}]\!]}
\newcommand{\op}[1]{{\rm{#1}}}

\def\Xint#1{\mathchoice
	{\XXint\displaystyle\textstyle{#1}}%
	{\XXint\textstyle\scriptstyle{#1}}%
	{\XXint\scriptstyle\scriptscriptstyle{#1}}%
	{\XXint\scriptscriptstyle\scriptscriptstyle{#1}}%
	\!\int}
\def\XXint#1#2#3{{\setbox0=\hbox{$#1{#2#3}{\int}$ }
		\vcenter{\hbox{$#2#3$ }}\kern-.6\wd0}}

\def\dashint{\Xint-}

\title[Matchings in metric spaces]{Matchings in metric spaces, the dual problem and calibrations modulo 2}

\author{Mircea Petrache}
\address{Max-Planck Institute for Mathematics, Vivatsgasse 7, 53111 Bonn, Germany}
\email{decostruttivismo@gmail.com}
\author{Roger Z\"{u}st}
\address{University of Bern, Mathematical Institute, Alpeneggstrasse 22, 3012 Bern, Switzerland}
\email{roger.zuest@math.unibe.ch}



\begin{document}

\begin{abstract}
We show that for a metric space with an even number of points there is a $1$-Lipschitz map to a tree-like space with the same matching number. This result gives the first basic version of an unoriented Kantorovich duality. The study of the duality gives a version of global calibrations for $1$-chains with coefficients in $\mathbb Z_2$. Finally we extend the results to infinite metric spaces and present a notion of ``matching dimension'' which arises naturally.
\end{abstract}

\maketitle


\section{Introduction}
Let $n \in \N$ and $X=\{x_1,\ldots,x_{2n}\}$ a set with $2n$ points equipped with a pseudometric $d$. A \emph{matching} on $X$ is a partition $\pi$ of $X$ into $n$ pairs of points, $\pi=\{\{x_1,x'_1\},\ldots,\{x_n,x'_n\}\}$. The set of all matchings on $X$ is denoted by $\cM(X)$. The main object of study in this work is the \emph{minimum matching problem} (cfr. \cite{edmonds} for a combinatorial analogue) for $d$, which is the following minimization:
\begin{equation}\label{matchingofd}
m(X,d) \defl \min_{\pi\in \cM(X)}\sum_{\{x,x'\} \in \pi} d(x,x')\ .
\end{equation}
The topic of the present work is the description of the dual problem for \eqref{matchingofd}. The interesting phenomenon is that dual objects are characterized by a special tree structure. A pseudometric space $(X,d)$ is said to be \emph{tree-like} if for any choice of points $x_1,x_2,x_3,x_4 \in X$,
\begin{align}
\nonumber
d(x_1,x_3) & + d(x_2,x_4) \\
\label{fourpointcond}
 & \leq \max\{d(x_1,x_2) + d(x_3,x_4), d(x_1,x_4) + d(x_2,x_3)\}\ .
\end{align}
$(X,d)$ is tree-like if and only if it can be realized as a subset of a metric tree (in case of a genuine pseudometric, we identify those points in $X$ with vanishing distance), see \cite{Za}, \cite{Bu} for finite spaces and \cite{D} for the general case. \emph{Metric trees} can be characterized as those metric spaces $(\tilde X,d)$ such that between any two points $x,y\in \tilde X$ there exists (up to reparametrization) a unique shortest curve from $x$ to $y$ of length $d(x,y)$. Throughout these notes we will also assume that metric trees are complete.

Our main basic observation is that:
\begin{Thm}
\label{mainthm}
For any pseudometric $d$ on a set $X$ of even cardinality, there is a tree-like pseudometric $D$ on $X$ with $D \leq d$ and $m(X,D) = m(X,d)$.
\end{Thm}

The metric $D$ that we construct has some additional properties. For example, there holds $\cH^1(T) = m(X,d)$ for the metric tree $T$ that is spanned by some minimal metric $D$ as in the theorem above, see Proposition~\ref{mainthmbis}. We develop three concepts as applications of Theorem~\ref{mainthm}, where the dual objects presented here give respectively a notion of unoriented Kantorovich duality, a notion of global calibrations modulo $2$ and a notion of matching dimension.

\subsection{Unoriented Kantorovich duality}
There is a very direct link between our duality result and a basic version of the so-called Kantorovich duality. More precisely we have in mind the following, by now classical, result (see \cite{kantorovich} for the originating idea, and see e.g.\ \cite[Lemma~2.2]{sandier} for a proof of this precise statement):

\begin{Thm}[Kantorovich duality]\label{kantorovich}
	Let $(X,d)$ be a metric space of cardinality $2n$. Let $\Pi=\{\{x_1^+,\ldots,x_n^+\},\{x_1^-,\ldots,x_n^-\}\}$ be a partition of $X$ into two $n$-ples of points. Then the following holds,
	\begin{equation}\label{calibration}
	\min_{\sigma\in S_n}\sum_{i=1}^nd\bigl(x_i^+,x_{\sigma(i)}^-\bigr) = \max\Biggl\{\sum_{i=1}^nf(x_i^+)-f(x_i^-) \ \Bigg|\begin{array}{l}f: X\to \R \\ \text{is }1\text{-Lipschitz} \end{array} \Biggr\}\ .
	\end{equation}
\end{Thm}

A matching $\{\{x_1^+, x_{\sigma(1)}^-\},\dots,\{x_n^+,x_{\sigma(n)}^-\}\}$ achieving the above minimum is sometimes called a \emph{minimal connection} corresponding to the partition $\Pi$ and in general a matching respecting this partition like in Theorem~\ref{kantorovich} is called an \emph{admissible connection} for $\Pi$. Let $M(\Pi, d)$ denote the length of the minimal connection. In another setting we may imagine that $X\subset\tilde X$ is a finite set in another metric space and we have two probability measures $\mu^+, \mu^-$ defined as
\begin{equation}\label{measuresatomic}
\mu^\pm \defl \frac{1}{n}\sum_{i=1}^n\delta_{x_i^\pm}\ .
\end{equation}
In this case we have
\[
M(\Pi,d)=W_1(\mu^+,\mu^-)\ ,
\]
where $W_1$ is the $1$-Wasserstein distance defined on probability measures (cfr.\ \cite{vil}, \cite{ags} and the references therein). By density considerations, if $\tilde X$ is Polish, then giving $W_1$ on measures of the type \eqref{measuresatomic} is the same as giving it on the whole set of probability measures on $\tilde X$.\\

Note that for a $1$-Lipschitz function $f:X\to\mathbb R$ satisfying \eqref{calibration} there holds
\begin{equation}\label{pushforwardminconn}
\sum_{i=1}^nf(x_i^+)-f(x_i^-) = \min_{\sigma\in S_n}\sum_{i=1}^n d_{\mathbb R}\Bigl(f\bigl(x_i^+\bigr),f\bigl(x_{\sigma(i)}^-\bigr)\Bigr) = M(\Pi, f^*d_{\mathbb R})\ .
\end{equation}
In view of \eqref{pushforwardminconn}, closely related to Theorem~\ref{kantorovich} is the following:

\begin{Thm}[Kantorovich duality, pullback formulation]\label{kantorovich2}
Let $(X,d)$ be a metric space of cardinality $2n$ and $\Pi=\{\{x_1^+,\ldots,x_n^+\},\{x_1^-,\ldots,x_n^-\}\}$ be a partition of $X$ into two $n$-ples of points. Then the following holds,
\begin{equation}
\label{calibration2}
M(\Pi,d) = \max\left\{M(\Pi,f^*d_{\mathbb R})\,|\, f: X\to\mathbb R \text{ is }1\text{-Lipschitz}\right\}\ .
\end{equation}
\end{Thm}

A slight reformulation of Theorem~\ref{mainthm} makes the analogy with the Kantorovich duality clear.

\begin{Thm}[unoriented Kantorovich duality]\label{mainthmbis2}
Let $(X,d)$ be a pseudometric space of cardinality $2n$. Then
\begin{equation}\label{minmaxmatch2}
m(X,d) = \max\biggl\{m(X,f^*d_T) \ \bigg|\begin{array}{l}f: X\to (T, d_T) \text{ is }1\text{-Lipschitz }\\\text{and }(T, d_T)\text{ is a metric tree}\end{array}\biggr\}\ .
\end{equation}
\end{Thm}
The important difference between this theorem and Theorem~\ref{kantorovich2} is that here the minimization is done amongst a wider class of competitors. The set $X$ has $\frac{(2n)!}{2^n n!}$ matchings and once we fix a partition $\Pi$ only $n!$ of them are admissible connections for it. Therefore there holds
\begin{equation*}
m(X,d) \le M(\Pi,d)\ ,
\end{equation*}
with a strict inequality in general. It might then look slightly surprising that, while on the one hand the minimum on the left side of \eqref{minmaxmatch2} is smaller than the corresponding minimum in the orientable problem, on the other hand in order to achieve the same number by the maximum on the right side we have to enlarge the space of $1$-Lipschitz maps competing for the dual problem, passing form $\mathbb R$ to general metric trees.\\

For the sake of concreteness we also formulate more explicitly a corollary of Theorem~\ref{mainthmbis2} in a special situation: 
\begin{Cor}\label{euclidean}
 Let $X\subset\mathbb R^n$ be a subset of even cardinality. Then there holds
 \[
   \min_{\pi\in \cM(X)}\sum_{\{x,x'\} \in \pi}|x - x'| = \max_{f, T}\min_{\pi\in \cM(X)}\sum_{\{x,x'\} \in \pi}d_T(f(x),f(x'))\ ,
 \]
 where the maximum is taken over all metric trees $(T,d_T)$ and all $1$-Lipschitz functions $f:\mathbb R^n\to T$.
\end{Cor}
For more properties of the maximizing couples $(f,T)$ see Proposition~\ref{mainthmbis}.
\subsection{Global calibrations modulo $2$}
In Section \ref{calibrations} we connect our result to the theory of calibrations, and give a natural answer in the first truly nontrivial case to the question of extending the notion of a calibration to the setting of the Plateau problem for chains with coefficients in a group. Let $(T,d_T)$ be a metric tree. A $1$-Lipschitz function $\rho : T \to \R$ is an \emph{orientation modulo 2 for $A \subset T$} if for any arc $[a,b] \subset T$ we have that the norm of the approximate differential (see \cite{F}) of $\rho|_{[a,b]}$ equals one, i.e.\ $J(\rho|_{[a,b]})(t) = 1$, for $\cH^1$-a.e.\ $t \in [a,b] \cap A$. Such orientations for $T$ are given for example by the distance functions $t \mapsto d_T(p,t)$ for any $p \in T$. 

Let $(\tilde X,d)$ be a metric space. As defined in \cite{depauwhardt}, the set $\cR_1(\tilde X,\Z_2)$ of rectifiable $1$-chains modulo $2$ is composed of chains $\curr \Gamma$, where $\Gamma$ is some $\cH^1$-rectifiable set $\Gamma \subset \tilde X$ (there holds $\curr \Gamma = \curr {\Gamma'}$ if and only if $\cH^1(\Gamma \Delta \Gamma') = 0$). The mass of $\curr \Gamma$ is given by $\bM(\curr \Gamma) = \cH^1(\Gamma)$ and assumed to be finite. We then associate to $\curr{\Gamma}$ a $\cH^1$-integrable coefficient map $g_{\curr{\Gamma}} : \tilde X \to \Z_2$ with the property that $g_{\curr{\Gamma}}(x) = 1$ if $x \in \Gamma$ and $g_{\curr{\Gamma}}(x) = 0$ otherwise. Applying the general framework \cite{depauwhardt} of flat chains with coefficients in a normed Abelian group $G$, this allows to define a group structure on $\cR_1(\tilde X, \Z_2)$. In the present case, $G=\Z_2$ and the sum $\curr{\Gamma_1} + \curr{\Gamma_2}$ is identified with $ \curr{\Gamma_1 \Delta \Gamma_2}$, since the coefficient map identity $g_{\curr{\Gamma_1} + \curr{\Gamma_2}} = g_{\curr{\Gamma_1}} + g_{\curr{\Gamma_2}}$ is interpreted for functions with values in $\Z_2$. If $f : \tilde X \to \R$ is Lipschitz we can define its action on $\curr \Gamma$ as follows. Fix some countable parametrization $\gamma_i : K_i \subset \R \to \gamma_i(K_i) \subset \Gamma$, i.e.\ $K_i$ is compact, the images $\gamma_i(K_i)$ are pairwise disjoint, all $\gamma_i$ are bi-Lipschitz and $\cH^1(\Gamma \setminus \cup_i \gamma_i(K_i)) = 0$. Then we define
\[
\curr \Gamma(df) \defl \sum_i \int_{K_i} |(f \circ \gamma_i)'(t)| \, d\cH^1(t) \ .
\]
It is not hard to check that this definition does not depend on the parametrization and on the choice of the set $\Gamma$ representing $\curr \Gamma$ as above. Further, $\curr \Gamma(df) \leq \Lip(f) \bM(\curr \Gamma)$ and if $f \in C^1(\R^n)$ and $\gamma : [0,1] \to \R^n$ is Lipschitz and injective, then $\gamma_\#\curr{0,1}(df) = \int_0^1 |df(\gamma'(t))|\, dt$, justifying the use of $df$ in the definition of this action. In contrast to chains with coefficients in $\Z$, this action is not linear. For Lipschitz functions $f,g$ and $C,C' \in \cR_1(\tilde X,\Z_2)$ there holds, $C(d(f + g)) \leq C(df) + C(dg)$ and $(C + C')(df) \leq C(df) + C'(df)$, with strict inequalities in general.

A Lipschitz chain $C \in \cL_1(\tilde X,\Z_2)$ is given by a finite sum $\sum_{i=1}^n \gamma_{i\#} \curr{0,1}$ for Lipschitz curves $\gamma_i : [0,1] \to \tilde X$. The boundary of $C$ is defined to be 
\[
\partial C \defl \sum_{i=1}^n \curr{\gamma_i(1)} + \curr{\gamma_i(0)} \ ,
\]
see \cite[Theorem~4.2.1]{depauwhardt}. This shows that such boundaries are composed of an even number of points. From \cite[Theorem~4.3.4]{depauwhardt} it follows that the same is true for any $C \in \cR_1(\tilde X,\Z_2)$ with finite boundary mass. On the other side, if $\tilde X$ is Lipschitz path connected, any collection of an even number of points in $\tilde X$ is the boundary of some Lipschitz chain.

\begin{Prop}
	\label{calibintro}
Let $(\tilde X,d)$ be a geodesic metric space and let $\curr X$ be a $0$-boundary modulo $2$ in $\tilde X$ (i.e.\ $X$ is a subset of even cardinality). Let $f : \tilde X \to T$ be a $1$-Lipschitz map into a metric tree $(T,d_T)$ with $m(X,d) = m(X, f^* d_T)$ and $\rho$ an orientation modulo $2$ for $T$. Then for any $C \in \cR_1(\tilde X,\Z_2)$ with $\partial C = \curr X$ there holds
\[
m(X,d) \leq C(d(\rho \circ f)) \leq \bM(C) \ ,
\]
with equalities if and only if $C=\sum_{i=1}^n\curr{x_i, y_i}$ where $[x_i,y_i]$ are geodesic segments and $\{\{x_i,y_i\},1\le i\le n\}$ is a minimal matching for $(X,d)$.
\end{Prop}

We then may define:
\begin{Def}[global calibrations modulo $2$]
 Let $(\tilde X,d)$ be a geodesic metric space and let $\curr X$ be a $0$-boundary modulo $2$ in $\tilde X$. The differential $d(\rho \circ f)$ for $f,\rho$ like in Proposition~\ref{calibintro} is called a global calibration modulo $2$ for $\curr{X}$.
\end{Def}
For the proof of Proposition~\ref{calibintro} and more properties of global calibrations modulo $2$ see Theorem~\ref{mainthmcal}. As a link to classical results, we include Proposition~\ref{usualcal} which is the analogue of Proposition~\ref{defcal2justif} valid for usual calibrations. See Subsection~\ref{gencalibrations} for references to the existing literature. Three directions for generalizations are briefly discussed in the remarks at the end of Section~\ref{calibrations}.

\subsection{Matching dimension}
As a concrete application of our new global duality result for matchings, we prove an incompressibility property for minimum matchings. If we have $k$ points constrained in a $n$-dimensional cube of side $1$, then we show that the maximal total length of the minimum matching segments behaves like $k^{\frac{n-1}{n}}$. This result uses the properties of the tree we construct in connection with the matching number and the coarea-formula. See Proposition~\ref{maxmatch} for this result. An analogy with this Euclidean case justifies in particular to define the \emph{matching dimension} of a metric space. \\

\textbf{Acknowledgements. }Corollary \ref{euclidean} answers a question posed by Tristan Rivi\`ere, to whom go our thanks. The first author was supported by the Fondation des Sciences Math\'ematiques de Paris and the second author was supported by the Swiss National Science Foundation.

\section{Calibrations modulo 2}\label{calibrations}
\subsection{Calibrations for integral chains}\label{integralchains}
We recall here the setting of the theory of calibrations (see \cite{harveylawson}, \cite{federcal}). The following is a simple proof that the shortest oriented curve connecting two points $a,b\in\mathbb R^n$ is the oriented segment $[a,b]$. Let $\alpha = df$ be the constant differential form obtained as the exterior derivative of the linear function $f(x) = \langle x, \tau \rangle$, where $\tau = \frac{a-b}{|a-b|}$.
Then for any other Lipschitz curve $\gamma$ from $a$ to $b$ we have
\begin{equation}\label{calibreineq}
 \op{lenght}([a,b]) = \int_{[a,b]}\alpha = \int_{\gamma}\alpha \le \op{lenght}(\gamma)\ ,
\end{equation}
where we used the fact that $\alpha$ is $d$-closed, i.e.\ $d\alpha=0$, for the middle equality and for the remaining equality and inequality we used the fact that $\alpha$ has comass $1$, i.e.\ $\|\alpha\| = \max_{v \in S^{n-1}} \alpha(v) \leq 1$, with equality $\alpha(\tau) = 1$.

More in general, we may apply the same method for minimizers of the following problem. Let $\{\{x_1^+,\dots,x_n^+\},\{x_1^-,\dots,x_n^-\}\} = \{X^+,X^-\}$ be a partition of a set of $2n$ points in a connected Riemannian manifold $\tilde X$ and set $\curr{X^{\pm}} \defl \sum_{i=1}^n(\curr{x_i^+} - \curr{x_i^-}) \in \cR_0(\tilde X, \Z)$. Consider then
\begin{equation}\label{minmod0}
 \F_{\mathbb Z}(\curr{X^{\pm}}) \defl \inf\bigl\{\bM(C)\ |\ C \in \cR_1(\tilde X, \Z) \text{ and }\partial C=\curr{X^{\pm}} \bigr\}\ .
\end{equation}
Here $\cR_m(\tilde X,\Z)$ is the space of rectifiable chains of dimension $m$ with coefficients in $\Z$ as defined in \cite{depauwhardt}. $\cR_m(\tilde X,\Z)$ agrees with the space of $m$-dimensional integer rectifiable currents as defined in \cite{F}. In particular, $C \in \cR_m(\tilde X,\Z)$ acts on differential $m$-forms and its boundary has the defining property $\partial C(\omega) = C(d\omega)$. Using this duality with smooth $m$-forms \eqref{calibreineq} can be generalized to prove the minimality of certain chains in $\cR_m(\tilde X,\Z)$ as well. A \emph{calibration} of dimension $m$ is a comass $1$ closed $m$-form. This is one of the most robust tools for testing the minimality of oriented submanifolds. For more precise definitions and extensions see \cite{harveylawson}.

\subsection{Plateau problem for chains modulo \textit{p}}\label{platmodp}
Here and in the rest of this section we consider a set $\{x_1,\ldots,x_{2n}\}=X\subset\tilde X$ of cardinality $2n$ where $\tilde X$ is a geodesic metric space and $X$ has the induced metric $d$. The condition $\# X=2n$ implies that $\curr{X} \defl \sum_{i=1}^{2n}\curr{x_i}$ is the boundary of some $1$-chain with coefficients in $\mathbb Z_2$. In our setting we recall that $k$-dimensional chains with coefficients in a normed Abelian group $G$ are the completion for the so-called \emph{flat distance} of the set of finite sums of Lipschitz singular $k$-simplices with multiplicities in $G$. See \cite{depauwhardt} for more details. We consider the $1$-dimensional unoriented Plateau problem analogous to the one of the previous section:
\begin{equation}\label{minmod2}
\F_{\mathbb Z_2}(\curr X) \defl \inf\bigl\{\bM(C)\ |\ C \in \cR_1(\tilde X, \Z_2)\text{ and }\partial C=\curr{X}\bigr\}\ ,
\end{equation}
We encourage the interested reader to consult \cite{F} and \cite{ambrosiowenger}, \cite{white} for results on the solution of the Plateau problem and for the case of $k$-chains with coefficients in a normed Abelian group like $\mathbb Z_p$. We just mention here that in our case $p=2$ the minimum in \eqref{minmod2} is realized and equal to $m(X,d)$. Moreover, minimizers $C$ are precisely chains of the form
\[
 C = \sum_{\{x,y\}\in\pi} \curr{x,y}\ ,\quad \text{where }\left\{\begin{array}{l}\pi\text{ is a minimizer of \eqref{matchingofd} and } \curr{x,y}\\ \text{is the }1\text{-chain corresponding to}\\\text{some geodesic segment }[x,y]\subset\tilde X.\end{array}\right.
\]

Contrary to the case of integral chains, there is no linear duality with $1$-forms for $1$-chains with coefficients in $\mathbb Z_p$. Therefore if we want to find a replacement for calibrations allowing to test minimality like in \eqref{calibreineq} a different object must be found.\\

Some partial extension of the duality method was already considered in \cite{morgan} for chains with coefficients in $\mathbb Z_p$ in Euclidean spaces $\mathbb R^n$. The observation there is that imposing extra local conditions on the calibration forms and some multiplicity bounds on projections for the minimizing objects has the effect of reducing the study of the minimization to a situation similar to the integer coefficient case. For some related negative results see also \cite{young}.\\

As explained in \cite{morgan} and in Examples~\ref{closuremod2}, in general having only a local condition on calibrations will not insure \emph{global minimality} of calibrated chains with coefficients in $\Z_2$. Our result gives a natural and optimal notion of calibrations for $1$-chains with coefficients in $\mathbb Z_2$ by capturing the nonlocal phenomena. We will see below (in Remark \ref{calmod3}) that different ideas are needed for a similar natural notion in the case of other groups, e.g.\ $\mathbb Z_p$, $p>2$.

\subsection{Global calibrations modulo 2}
We now describe an extension of Theorem~\ref{mainthmbis2} which allows to build a solid analogy with the result of Subsection~\ref{integralchains}.\\

Given a closed set $A\subset\tilde X$ and a set $X\subset\tilde X$ of even cardinality, we say that $A$ is a \emph{$\mathbb Z_2$-cut} of $X$ if at least one of the connected components of $\tilde X\setminus A$ contains an odd number of points in $X$. Then denote
\begin{equation}\label{defcut2}
\op{Cut}_{\mathbb Z_2}(A,X) \defl \#\biggl\{\begin{array}{c}\text{connected components }A' \text{ of } A\\\text{that are }\mathbb Z_2\text{-cuts} \end{array}\biggr\}\ .
\end{equation}
For a Lipschitz function $\varphi:\tilde X\to\mathbb R$ we define
\begin{equation}\label{deflev2}
 \op{lev}_{\mathbb Z_2}(\varphi,X) \defl \int_{\R}\op{Cut}_{\mathbb Z_2}(\varphi = t,X) \, dt\ .
\end{equation}
We then consider the following real number:
\begin{equation}\label{calimod2}
 \Lev_{\mathbb Z_2}(X) \defl \sup\Bigl\{\op{lev}_{\mathbb Z_2}(\varphi)\, \Big| \, \varphi:\tilde X\to \mathbb R \text{ is }1\text{-Lipschitz}\Bigr\}\ .
\end{equation}
For a map $f : X \to T$ defined on an even cardinality metric space $X$ into a tree, define 
\begin{equation}\label{defamxd}
 A_{X} \defl \bigcup\left\{[f(x),f(y)] \left| \begin{array}{l}\{x,y\}\text{ appears in some}\\ \text{ minimal matching of } (X,d) \end{array}\right.\right\}\ .
\end{equation}
See Proposition~\ref{mainthmbis} for some properties of this set. We then have the following result.

\begin{Thm}\label{mainthmcal}
Let $(\tilde X,d)$ be a geodesic metric space, $X=\{x_1,\ldots,x_{2n}\} \subset \tilde X$ be some even cardinality subset and $\varphi : \tilde X \to \R$ be a $1$-Lipschitz function. Consider the following statements:
\begin{enumerate}
	\item $\varphi = \rho \circ f$ for $1$-Lipschitz maps $f : \tilde X \to T$, $\rho : T \to \R$ where $(T,d_T)$ is a metric tree, $\rho$ is an orientation modulo $2$ for $A_X$ and $m(X,d) = m(X,f^*d_T)$.
	\item For any $1$-chain $C \in \cR_1(\tilde X,\Z_2)$ with $\partial C=\curr{X}$ there holds $C(d\varphi) \ge \F_{\Z_2}(\curr X)$.
	\item $\op{lev}_{\mathbb Z_2}(\varphi,X) = \op{Lev}_{\mathbb Z_2}(X)$.
\end{enumerate}
The following implications hold: $(1) \Rightarrow (2)$. If $\pi_1^{\Lip}(\tilde X) = 0$ then $(2) \Rightarrow (1)$. If $H_1(\tilde X) = 0$ or $H_1^{\Lip}(\tilde X) = 0$ then $(1) \Leftrightarrow (3)$. In particular if $\pi_1^{\Lip}(\tilde X) = 0$ then all three statements are equivalent.

Moreover, If $H_1(\tilde X) = 0$ or $H_1^{\Lip}(\tilde X) = 0$, then
\[
m(X,d) = \F_{\Z_2}(\curr X) = \Lev_{\mathbb Z_2}(X) \ .
\]
\end{Thm}

This theorem implies Proposition~\ref{calibintro} in the introduction and will be proved in Subsection~\ref{proofofmainthmcal}. The following examples show that the implications $(2) \Rightarrow (1)$ and $(3) \Rightarrow (1)$ do not hold for $\tilde X = S^1$.

\begin{Expls}
$(2) \nRightarrow (1)$: Consider $S^1 \subset \mathbb C$ with the intrinsic metric $d$ and let $X \defl \{i,-i\}$. Define $\varphi : S^1 \to \R$ by $\varphi(p) \defl d(1,p)$. Then $\varphi(1) \neq \varphi(-1)$. Clearly, $m(X,d) = C(d\varphi) = \pi$ for any of the two chains $C \in \cR_1(S^1,\Z_2)$ with $\partial C = \curr X$. But if $\rho \circ f = \varphi$ with $d_T(f(i),f(-i)) = \pi$ as in (1), $f$ would map both arcs in $S^1$ connecting $i$ with $-i$ isometrically to some geodesic segment of length $\pi$. But then $f(1)$ and $f(-1)$ are both equal to the midpoint of this segment contradicting $\varphi(1) \neq \varphi(-1)$.

$(3) \nRightarrow (1)$: For any $X \subset S^1$ consisting of two different points and any Lipschitz function $\varphi : S^1 \to \R$, there always holds $\op{lev}_{\Z_2}(\varphi,X) = 0$ since there is no connected set in $S^1$ that disconnects $X$. Hence, $\op{Lev}_{\Z_2}(X) = 0$ and any $1$-Lipschitz function achieves this maximum. Therefore $(3) \Rightarrow (1)$ doesn't hold.
\end{Expls}

The analogue of $\op{Cut}_{\mathbb Z_2}(A,X)$ and $\op{lev}_{\mathbb Z_2}(\varphi,X)$ for the minimization on integral chains like in Section~\ref{integralchains} is as follows. For a closed set $A \subset \tilde X$ and for a partition $\Pi=\{X^+,X^-\}$ of $X$ into two parts of equal cardinality, define the quantity
\[
 \op{Cut}_{\mathbb Z}(A,\Pi) \defl \left|\#(A\cap X^+) - \#(A\cap X^-)\right|\ .
\]
Then for a $1$-Lipschitz function $f:\tilde X\to \mathbb R$ define 
\begin{equation}\label{levz}
 \op{lev}_{\mathbb Z}(f,\Pi) \defl \int_{\R}\op{Cut}_{\mathbb Z}(\{f\le t\},\Pi) \, dt\le \F_{\mathbb Z}(\curr {X^\pm}) \ .
\end{equation}

\begin{proof}[Proof of \eqref{levz}]
Indeed, let $C$ be an integer rectifiable $1$-chain in $\tilde X$ with $\partial C = \curr{X^+} - \curr{X^-}$. Parametrize $C$ via the triple $[ \Gamma, \theta, \tau ]$ where $\Gamma$ is a $\cH^1$-rectifiable set, $\theta\in L^1(\Gamma,\cH^1)$ and has values a.e.\ in $\mathbb Z \setminus \{0\}$, and $\tau$ is a $\mathcal H^1$-measurable orienting vector field for $\Gamma$. Then via the area formula and using the fact that $f$ is $1$-Lipschitz there holds
\begin{align*}
 \int_{\mathbb R} \sum_{p\in f^{-1}(t)\cap \Gamma}|\theta(p)|\,dt & = \int_\Gamma J(f|_{\Gamma})(p)|\theta(p)|\,d\cH^1(p) \\
 & \leq \int_\Gamma|\theta(p)|\,d\cH^1(p) = \bM(C)\ .
\end{align*}
Since $\op{Cut}_{\Z}(\{f\le t\},\Pi)$ is bounded from above by $\sum_{p\in f^{-1}(t)\cap\Gamma}|\theta(p)|$ for a.e.\ $t\in\mathbb R$ it follows that $\op{lev}_{\mathbb Z}(f,\Pi) \leq \bM(C)$. Taking now the infimum over all such $C$ like in \eqref{minmod0} we conclude.
\end{proof}

If $\Lev_{\mathbb Z}(\Pi)$ is defined to be the supremum of $\op{lev}_{\Z}(f,\Pi)$ among all $f$ as above, then we see immediately that Theorem~\ref{kantorovich} states exactly that $\F_{\Z}(\curr {X^\pm}) = \Lev_{\mathbb Z}(\Pi)$. Calibrations like in Subsection~\ref{integralchains} appear via the following well-known fact, of which we provide a sketch of proof for the convenience of the reader.

\begin{Prop}\label{usualcal}
Let $\tilde X$ be a connected Riemannian manifold with $0 = H^1(\tilde X,\R)$ and $\Pi$ be some partition $\{\{x_1^+,\dots,x_n^+\},\{x_1^-,\dots,x_n^-\}\}$ of a finite subset $X$ of $\tilde X$. Let $C \in \cR_1(\tilde X, \Z)$ be an integer rectifiable $1$-chain with $\partial C = \curr{X^{\pm}}$ and $\bM(C) = \F_{\Z}(\curr{X^{\pm}})$. For a closed flat $1$-form $\alpha$ on $\tilde X$ the following are equivalent:
 \begin{enumerate}
  \item $\alpha$ is a calibration for $C$, i.e.\ $\alpha$ has comass $1$ and satisfies $C(\alpha)=\bM(C)$.
  \item $\alpha$ is a calibration for any minimizer $C$ as above.
  \item $\alpha=d f$ for some $1$-Lipschitz function $f : \tilde X \to \R$ for which as in \eqref{calibration}
  \[
  \min_{\sigma\in S_n}\sum_{i=1}^nd\bigl(x_i^+,x_{\sigma(i)}^-\bigr) = \sum_{i=1}^nf(x_i^+)-f(x_i^-) \ .
  \] 
  \item $\alpha=d f$ for some $1$-Lipschitz function $f : \tilde X \to \R$ realizing the equality $\op{lev}_{\mathbb Z}(f,\Pi) = \op{Lev}_{\mathbb Z}(\Pi)$.
 \end{enumerate}
\end{Prop}
\begin{proof}[Sketch of proof:]
We prove $(1)\Rightarrow(3)\Rightarrow(2)\Rightarrow(1)$. Note first that $(2)\Rightarrow(1)$ is trivial. To prove $(1)\Rightarrow (3)$ note
first that since $H^1(\tilde X,\R)=0$ a variant of de Rham's theorem implies that $\alpha = d\beta$ for some locally flat $0$-form $\beta$ \cite[Theorem~4.4]{GH}. Since any flat $0$-form in $\R^n$ can be represented by a Lipschitz function, see e.g.\ Theorem~5.5 and Theorem~5.12 of \cite{H}, we can write $\beta = f$ for some locally Lipschitz function $f: \tilde X \to\mathbb R$. The comass $1$ condition on $\alpha$ then reads
\[
1 \geq \op{esssup}_{x \in \tilde X} \|\alpha\| = \op{esssup}_{x \in \tilde X} \max\{|df_x(v)|\ |\ v \in T_x\tilde X, |v| = 1 \}
\]
and since $\tilde X$ is geodesic this translates to $f$ being $1$-Lipschitz. Recall that a chain $C$ with $\bM(C) = \F_{\Z}(\curr{X^{\pm}})$ is of the form $C = \sum_{i=1}^n \curr{x_{\sigma(i)}^-, x_i^+}$ for some $\sigma \in S_n$, where $[x_{\sigma(i)}^-, x_i^+]$ is a geodesic from $x_{\sigma(i)}^-$ to $x_i^+$. Then since $f$ is $1$-Lipschitz and $C(df) = \bM(C) = \sum_{i=1}^n d\bigl(x_i^+, x_{\sigma(i)}^-\bigr)$, using the fact that $\partial C(f)=C(df)$ by the definition of boundary for currents \cite{F}, there holds
 \begin{align*}
  \sum_{i=1}^n f(x_i^+) - f(x_i^-) & \leq \sum_{i=1}^n d\bigl(x_i^+, x_{\sigma(i)}^-\bigr) = C(df) =\partial C(f) \\
	 & = \sum_{i=1}^n f(x_i^+) - f(x_i^-) \ .
 \end{align*}
Thus the inequality above is an equality and (3) holds. An analogous reasoning gives $(3)\Rightarrow (2)$. We then similarly prove $(1)\Rightarrow(4)\Rightarrow(2)$ by noting that counting the number of level sets of a $1$-Lipschitz function $f$ crossed by a curve $C$ which is part of a minimal connection gives the highest value when $\nabla f$ is the orienting unit tangent vector field of $C$.
\end{proof}

Note the following analogue of the above proposition.

\begin{Prop}\label{defcal2justif}
 Let $\tilde X$ be a connected Riemannian manifold with $0=H_1(\tilde X)$ and let $X\subset\tilde X$ be an even cardinality set. 
 Let $C \in \cR_1(\tilde X,\Z_2)$ with $\partial C = \curr X$ and $\bM(C) = \op{Fill}_{\mathbb Z_2}(\curr{X})$. For a closed flat $1$-form $\alpha$ on $\tilde X$ consider the following statements:
 \begin{enumerate}
  \item $\alpha$ has comass $1$ and for a fixed $C$ as above, $C(\alpha) = \bM(C)$.
  \item $\alpha$ has comass $1$ and for any minimizer $C$ as above, $C(\alpha) = \bM(C)$.
  \item $\alpha = d(\rho \circ f)$, where $f : \tilde X \to T$ is a $1$-Lipschitz map into a finite tree $(T,d_T)$ such that $m(X,d) = m(X, f^* d_T)$ and $\rho$ is an orientation for $A_{X}$ defined in \eqref{defamxd}.
  \item $\alpha = d \varphi$ for some $1$-Lipschitz function $\varphi : \tilde X \to \R$ realizing the equality $\op{lev}_{\mathbb Z_2}(\varphi,X)=\op{Lev}_{\mathbb Z_2}(X)$. 
 \end{enumerate}
 Then $(4) \Leftrightarrow (3)\Rightarrow (2) \Rightarrow (1)$.
\end{Prop}
\begin{proof}
$(3)\Rightarrow (2)$ is a particularization of Proposition~\ref{mainthmbis}(1) (stated in the next section), which gives the slightly more precise information that any $f$ like in $(3)$ is actually an isometry when restricted to the segments forming $C$. The implication $(2)\Rightarrow (1)$ is trivial. The implication $(3)\Leftrightarrow(4)$ follows directly from Theorem~\ref{mainthmcal}.
\end{proof}

In general we don't have the implications $(1) \Rightarrow (2)$ and $(2) \Rightarrow (3)$ as the following examples demonstrate.

\begin{Expl}[loss of information in conditions (1) and (2)]
	\label{closuremod2}
$(1) \nRightarrow (2)$: Let $X$ be the collection of the four points $p_1 = (1,1), p_2 = (1,-1), p_3 = (-1,-1)$ and $p_4 = (-1,1)$ in $\R^2$. There are exactly two minimizers with boundary $\curr X$, namely $C_1 = \curr{p_1,p_4} + \curr{p_2,p_3}$ and $C_2 = \curr{p_1,p_2} + \curr{p_3,p_4}$. Obviously, $C_2(dx) = 0$ and $C_1(dx) = 4 = \bM(C_2)$. (1) holds for $C_1$ but not for $C_2$, hence we don't have (2).

$(2) \nRightarrow (3)$: Let $\tilde X = \B^2(0,1) \subset \R^2$ and select $X = \{(1,0),(-1,0)\}$. The function $\varphi(p) = |p|$ is $1$-Lipschitz and $C(d\varphi) = 2 = \bM(C)$ for the unique minimizer $C = \curr{(1,0),(-1,0)}$ with boundary $\curr X$. But none of the level sets $\varphi^{-1}(t)$ disconnects the two points of $X$. Hence $\op{lev}_{\Z_2}(\varphi,X) = 0 < 2 = \op{Lev}_{\Z_2}(X)$ by the $(1)\Leftrightarrow(3)$ part of Theorem~\ref{mainthmcal}. The equivalence of (3) and (4) in Proposition~\ref{defcal2justif} now implies that (3) doesn't hold.
\end{Expl}

\subsection{Generalizations}\label{gencalibrations} Let $G$ be a complete normed Abelian group. Under the hypothesis that $\{g\in G\ |\ |g| \leq M\}$ is compact for all $M > 0$ it follows by the lower-semicontinuity of mass under flat convergence that the Plateau problem is solvable for flat $G$-chains in $\mathbb R^n$, \cite[Corollary~7.5]{fleming}, i.e.\ given $S \in \cF_{k-1}(\R^n,G)$ with $\bM(S) < \infty$, $\partial S = 0$ and compact support, then there is a $T \in \cF_k(\R^n,G)$ with $\partial T = S$, compact support and
\[
\bM(T) = \inf\{\bM(C)\ |\ C \in \cF_k(\R^n,G) \text{ and } \partial C = S \}\ .
\]
Furthermore, the flat chains considered above are actually rectifiable if there exists no nonconstant Lipschitz path $\gamma:[0,1]\to G$ (see \cite[Theorem~8.1]{white}). Hence the Plateau problem is solvable for rectifiable chains with coefficients in $G$. These two conditions on $G$ stated above are true for the discrete groups $G=\mathbb Z$ and $G=\mathbb Z_2$ for which we have a duality as above, but also for more general discrete groups like $\mathbb Z_p$ or $\mathbb Z^d$, for $\mathbb Z$ endowed with the $p$-adic norm and for $\mathbb R^d$ or $S^1$ endowed with the snowflaked distance $d_\alpha(x,y) \defl |x-y|^\alpha$, $\alpha\in\ ]0,1[$. We include here some remarks about global duality questions for the problem of minimizing the length of $1$-chains in some of the cases which remain open.

\begin{Rem}[global calibrations modulo $p$ for $p\ge 3$]\label{calmod3}
Already for the coefficient group $\Z_3$ the same approach as for $\Z_2$ doesn't work anymore and $1$-Lipschitz maps into metric trees are not a suitable Ersatz for calibrating $1$-dimensional chains with coefficients in $\Z_3$.
To see this, let $X = \{\omega_1,\omega_2,\omega_3\}$ be the third roots of unity and $\curr X$ be the $0$-dimensional chain with coefficient $1 \in \Z_3$ supported on $X$. In this case the unique minimal filling of $\curr X$ is the cone $C$ spanned by $\curr X$ forming a triple junction at the origin. Assume $f : \mathbb R^2\to T$ is a $1$-Lipschitz map and ``calibrates'' the unique minimizer $C$, with level sets as sketched in Figure \ref{compatiblez3}. Then the subtree of $T$ spanned by $f(\omega_1),f(\omega_2)$ and $f(\omega_3)$ is equal to some possibly irregular tripod with total length at most $3\sqrt3/2$ where the maximum is achieved by the regular tripod with ``leg-length'' $\sqrt 3/2$ (Note that $|\omega_i - \omega_j| = \sqrt 3$ for $i \neq j$). Hence we would have
\[
\bM(f_\# C) \le \frac{3\sqrt3}{2} < 3 =\bM(C) \ ,
\]
contradicting the calibration property. This indicates that if a satisfactory notion of ``calibration modulo $3$'' (or modulo $p$ for $p>2$) can be created at all, then it must be quite different than the one appearing for $p=2$ here.

The same space $X$ also illustrates a fundamental difference between the Plateau problem for $1$-dimensional chains with coefficients in $\Z_p$, $p \geq 3$, and the one with coefficients in $\Z$ or $\Z_2$. Let $\iota : X \to T$ and $\kappa : X \to H$ be isometric embeddings into a metric tree $T$ and a Hilbert space $H$ respectively. Then
\[
\F_{\Z_3}(\iota_\# \curr X) = \frac{3 \sqrt{3}}{2} < 3 = \F_{\Z_3}(\kappa_\# \curr X) \,.
\]
Therefore the minimal filling length of the $\Z_3$-chain $\curr X$ depends on the ambient space, whereas the minimal filling length for $0$-dimensional chains with coefficients in $\Z$ or $\Z_2$ only depends on the pairwise distances between its supporting points (under the condition that the ambient space is geodesic).

\end{Rem}

\begin{figure}[htbp!]
	\centering
	\resizebox{0.5\textwidth}{!}{\includegraphics[width=0.5\textwidth]{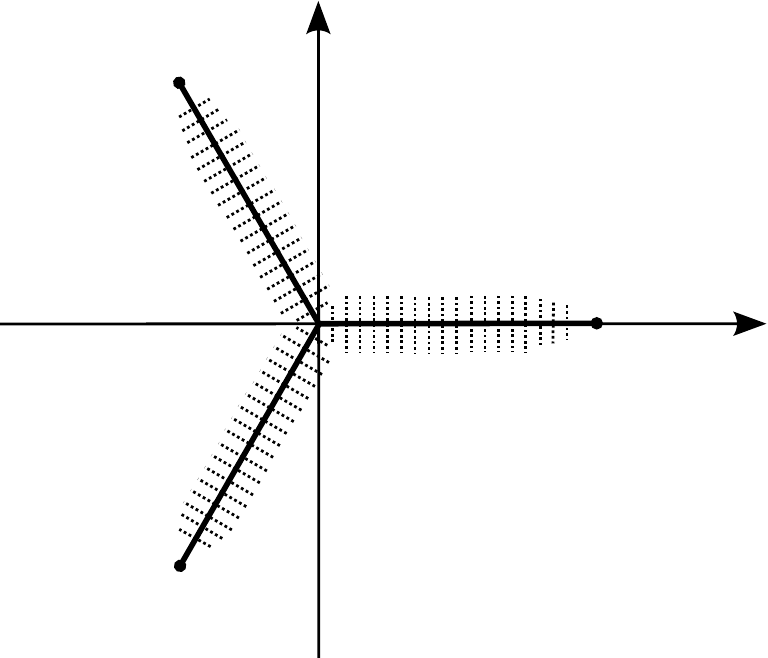}}
	\caption{We depict a solution $C$ of the Plateau problem for $1$-chains with coefficients in $\mathbb Z_3$. If $df$ were to describe a ``calibration'' similar to the ones described for $\mathbb Z_2$, then the level sets of $f$ near $\spt(C)$ would be given by the above transverse lines.}\label{compatiblez3}
\end{figure}

\begin{Rem}[calibrations for the coefficient group $(\mathbb R, d_\alpha)$]
 The minimization of mass for $1$-chains with coefficients in $\mathbb R$ endowed with the norm $d_\alpha(x,y)\defl|x-y|^\alpha$, $\alpha\in \ ]0,1[$, is exactly the same as the so-called \emph{branched optimal transport problem} or \emph{irrigation problem}. In that case a possible starting point for a duality theory is represented by the so-called \emph{landscape function} introduced in \cite{santambrogio}, see also \cite{xia}. For the description of this function and for further references we refer to these two papers.
\end{Rem}

\begin{Rem}[calibrations with coefficients in $\mathbb Z^d$ with different norms]
 A duality theory with coefficients in $\mathbb Z^d$ was introduced in \cite{marchesemassaccesi} for a different reason, and constant-coefficient calibrations were considered. In that case, if the norm on $\mathbb Z^d$ is symmetric enough the situation is analogous to the classical case $d=1$, and a duality between forms with values in $\mathbb R^d$ (also interpretable as $d$-tuples of forms and $\mathbb Z^d$-valued $1$-chains) is present. The question is relevant in crystallography problems \cite{garroni}.
\end{Rem}

\section{Proof of the main theorems}\label{proofs}
\subsection{Proof of Theorem~\ref{mainthm}}
To simplify the notation we write $1,2,\dots,2n$ for the points in $X$. The set of all matchings on $X$ is denoted by $\cM(X)$. 
The \emph{matching number} of some $\pi \in \cM(X)$ with respect to the metric $d$ is defined by
\[
m(\pi,d) \defl \sum_{\{i,j\} \in \pi } d(i,j) \ .
\]
The \emph{matching number} of $d$ then is
\[
m(X,d) \defl \min_{\pi \in \cM(X)} m(\pi,d)\ ,
\]
and a \emph{minimal matching} is some $\pi \in \cM(X)$ for which this minimum is achieved. The set of all minimal matchings is denoted by $\cM(X,d)$. We will write $\{i,j\} \in \cP(d)$ if there is a minimal matching $\pi \in \cM(X,d)$ with $\{i,j\} \in \pi$.

We denote with $\cD$ the set of pseudometrics $d'$ on $X$ with $d' \leq d$ and $m(X,d') = m(X,d)$. Given two pseudometrics $d_1$ and $d_2$ on $X$ we can associate a distance $\delta(d_1,d_2)$ by
\[
\delta(d_1,d_2) \defl \max_{i,j \in X} |d_1(i,j) - d_2(i,j)|\ .
\]
It is easy to check that $(\cD,\delta)$ is a compact metric space and the function $w : \cD \to \R$ given by
\[
w(d') \defl \sum_{i \neq j} d'(i,j)
\]
is continuous. Hence $w$ attains its minimum at some $D \in \cD$. The goal will be to show that $D$ is tree-like. By using a compactness argument like this, $D$ may be a pseudometric even if we started with a genuine metric $d$. This actually does not depend on the particular way of constructing $D$ and we can't reformulate Theorem~\ref{mainthm} using only metrics instead of pseudometrics. Indeed, consider the following example:

\begin{Expl}
	Let $X = \{1,2,3,4,5,6\}$ and $0 < \epsilon \leq 2$. For $i < j$ set,
	\[
	d(i,j) \defl
	\left\{ \begin{array}{ll}
	2 & \text{ if } 1 \leq i < j \leq 4 \ ,\\
	1 & \text{ if } 1 \leq i \leq 4 \text{ and } j= 5,6 \ , \\
	\epsilon & \text{ if } i=5, \, j=6 \ .
	\end{array} \right.
	\]
	Let $D$ be a tree-like pseudometric on $X$ with $D \leq d$ and $m(X,D) = m(X,d)$. By symmetry it is easy to check that $m(X,d) = 4$ and some matching $\pi \in \cM(X)$ is in $\cM(X,d)$ if and only if $\{5,6\} \notin \pi$. This forces $D(i,j) = d(i,j)$ unless $\{i,j\} = \{5,6\}$. Because $D$ is tree-like,
	\begin{align*}
		2 + D(5,6) & = d(1,2) + D(5,6) = D(1,2) + D(5,6) \\
		& \leq \max\{D(1,5) + D(2,6), D(1,6) + D(2,5)\} \\
		& = \max\{d(1,5) + d(2,6), d(1,6) + d(2,5)\} = 2\ ,
	\end{align*}
	and hence $D(5,6) = 0$. So $D$ can't be strictly positive.
\end{Expl}

With the definition of $D$ as an element of $(\cD,\delta)$ minimizing $w$, we get:

\begin{Lem}
	\label{minimallem}
	There is a pseudometric $D \leq d$ on $X$ with the property that for any other pseudometric $D'$ on $X$ with $D' \leq D$ and $D' \neq D$ we have $m(X,D') < m(X,D) = m(X,d)$.
\end{Lem}

Here is a first step which shows that there are many minimal matchings with respect to $D$. For simplicity we abbreviate $|ij| = D(i,j)$ and we denote 
\[
\cP(D) \defl \left\{\{i,j\}\ |\ \text{there exists some } \pi \in \cM(X,D)\text{ with }\{i,j\} \in \pi\right\} \ .
\]

\begin{Lem}
	\label{preplem}
	If $D$ is as in Lemma~\ref{minimallem}, then for all different points $i,j \in X$ we have $\{i,j\} \in \cP(D)$.
\end{Lem}

\begin{proof}
	The main obstacle in obtaining this result is the violation of the triangle inequality. Assume by contradiction that for different $i, j \in X$ we have $\{i,j\} \notin \cP(D)$ (this in particular implies that $\# X \geq 4$). Since pseudometrics are those symmetric functions $d': X \times X \to \R$ determined by the inequalities $d'(a,b) \ge 0$ and $d'(a,b) + d'(b,c) \ge d'(a,c)$, the only way in which making $|ij|$ smaller makes us exit the set of pseudometrics is if $|ij|=0$ or if there exists some $k \notin \{i,j\}$ for which we have
	\begin{equation*}
		|ij| + |jk| = |ik| \quad \text{or} \quad |ji| + |ik| = |jk|\ .
	\end{equation*}
	We write $[a,b] \subset [c,d]$ for some not necessarily different $a,b,c,d \in X$ if
	\[
	|ca| + |ab| + |bd| = |cd|\ .
	\]
	The following fact is easy to check:
	\begin{equation}
	\label{intervalinclusion}
	[a,b] \subset [c,d] \quad \text{implies} \quad [a,b] \subset [a,d] \quad \text{and} \quad [a,b] \subset [c,b] \ .
	\end{equation}
\textbf{Step 1.} \emph{If $|ij|>0$, $[i,j] \subset [k,l]$ and $\{k,l\} \in \cP(D)$, then $\{i,j\} \in \cP(D)$.} This will be contradicting our assumption that $\{i,j\} \notin \cP(D)$. We now pass to prove it. If $\{k,l\} = \{i,j\}$ there is nothing to show, so we assume that $\{k,l\} \neq \{i,j\}$. We first show that $\{k,j\} \in \cP(D)$ (similarly we can show $\{i,l\} \in \cP(D)$). If $j=l$ we already have $\{k,j\}\in \cP(D)$, so we assume that $j \neq l$. Note that $k \neq j$ because $|kj| = |ki| + |ij| > 0$ by \eqref{intervalinclusion} and by the assumption $|ij| > 0$. Since $\{k,l\} \in \cP(D)$, there is a $\pi \in \cM(X,D)$ with $\{k,l\} \in \pi$. Since $\pi$ is a matching, there is some $j' \in X \setminus \{j,k,l\}$ with $\{j,j'\} \in \pi$. By the minimality of $m(\pi,D)$ we obtain
\begin{equation}
	\label{optimalest}
	|kl| + |jj'| \leq \min\{|kj| + |lj'|,|kj'| + |lj| \}\ .
\end{equation}
Otherwise we could replace the pairs $\{k,l\}$, $\{j,j'\}$ in $\pi$ by $\{k,j\}$, $\{l,j'\}$ or $\{k,j'\}$, $\{l,j\}$ to obtain a new matching with a smaller matching number, but this is not possible. Because of \eqref{intervalinclusion} we have $[k,j] \subset [k,l]$ which together with \eqref{optimalest} implies
\begin{align*}
	|kl| + |jj'| \leq |kj| + |lj'| \leq |kj| + |lj| + |jj'| = |kl| + |jj'|\ .
\end{align*}
This means that both inequalities are actually equalities and in particular $|lj'| = |lj| + |jj'|$. Hence,
\begin{equation*}
	|kj|+|lj'| = |kj| + |lj| + |jj'| = |kl| + |jj'|\ ,
\end{equation*}
because $[k,j] \subset [k,l]$. So, by replacing the pairs $\{k,l\}$, $\{j,j'\}$ in $\pi$ with $\{k,j\}$, $\{l,j'\}$ we obtain a matching $\pi'$ with the same, and therefore minimal, matching number. This implies $\{k,j\} \in \cP(D)$ as desired. Now if $k=i$ we have directly $\{i,j\} \in \cP(D)$, whereas if $k \neq i$ we have by \eqref{intervalinclusion} that $[i,j] \subset [k,j]$. Repeating the arguments above with these two intervals we obtain that $\{i,j\} \in \cP(D)$, contradicting our assumption.

\textbf{Step 2.} \emph{Proof of the lemma in case $|ij| > 0$.} For such $i,j$, define the following set of pairs
\[
\cP_{i,j} \defl \left\{\{k,l\} \in 2^X\ |\ [i,j] \subset [k,l] \text{ or } [j,i] \subset [k,l] \right\}\ .
\]
Note that $[j,i] \subset [k,l]$ is equivalent with $[i,j] \subset [l,k]$, so $\cP_{i,j}$ is well defined. For $\epsilon > 0$ define
\[
D_\epsilon(a,b) \defl|xy|_\epsilon \defl \left\{\begin{array}{ll}|ab| - \epsilon&\text{ if }\{a,b\} \in \cP_{i,j}\ ,\\ |ab|&\text{ else}\ .\end{array}\right.
\]
Since $|ij| > 0$ (and hence also $|kl| > 0$ if $\{k,l\} \in \cP_{i,j}$) we can assume that $\epsilon$ is small enough such that $D_\epsilon \geq 0$ and that $\{k,l\} \notin \cP(D)$ implies $\{k,l\} \notin \cP(D_\epsilon)$ and all strict triangle inequalities for $D$ are also strict triangle inequalities for $D_\epsilon$. By the definition of $D$ we have two possibilities: either $m(X,D_\epsilon) < m(X,D)$ or some triangle inequality of $D_\epsilon$ is violated. So assume first that $m(X,D_\epsilon) < m(X,D)$. By the choice of $\epsilon$ there is some $\{k,l\} \in \cP_{i,j}$ for which $\{k,l\} \in \cP(D)$, but this implies $\{i,j\} \in \cP(D)$ by Step 1, which in turn gives a contradiction to the initial assumption $\{i,j\}\notin \cP(D)$. Assume instead that some triangle inequality is violated for $D_\epsilon$. By the choice of $\epsilon$ this means that there are different $a,b,c \in X$ with 
\begin{equation}
\label{epsilonmetric}
|ab| + |bc| = |ac| \quad \text{and} \quad |ab|_\epsilon + |bc|_\epsilon < |ac|_\epsilon\ .
\end{equation}
In order for the strict inequality to hold at least one of the pairs $\{a,b\}$, $\{b,c\}$ needs to be in $\cP_{i,j}$. We assume $\{a,b\} \in \cP_{i,j}$, the other cases being similar. Up to switching $i$ and $j$ if necessary we may assume $[i,j] \subset [a,b]$. Now $\{a,b\} \in \cP_{i,j}$ implies
\begin{align*}
|ai| + |ij| + |jc| & \geq |ac| = |ab| + |bc| = |ai| + |ij| + |jb| + |bc| \\
 & \geq |ai| + |ij| + |jc|\ ,
\end{align*}
and hence $|ac| = |ai| + |ij| + |jc|$, which forces $\{a,c\} \in \cP_{i,j}$. This means that $|ac|_\epsilon = |ac| - \epsilon$ and $|ab|_\epsilon = |ab| - \epsilon$. In order to obtain the strict inequality in \eqref{epsilonmetric}, it is therefore necessary that $\{b,c\} \in \cP_{i,j}$. If $[i,j] \subset [b,c]$, then
\begin{align*}
|ac| & = |ab| + |bc| = (|ai| + |ij| + |jb|) + (|bi| + |ij| + |jc|) \\
& = (|ai| + |ij| + |jc|) + (|bi| + |ij| + |jb|) \geq |ac| + |ij| > |ac|\ ,
\end{align*}
since $|ij| > 0$ by assumption. If $[j,i] \subset [b,c]$, then
\begin{align*}
|ac| & = |ab| + |bc| = (|ai| + |ij| + |jb|) + (|ci| + |ij| + |jb|) \\
& \geq |ac| + 2|ij| > |ac|\ .
\end{align*}
Both of these options lead to a contradiction, as desired. Therefore the requirement $\{i,j\}\notin\cP(D)$ leads to a contradiction in case $|ij|>0$, and we conclude the proof of the Lemma in this case.

\textbf{Step 3.}\emph{ Proof of the lemma in case $|ij|=0$.} Pick any optimal $\pi \in \cM(X,D)$. Because $\pi$ is a matching and $\{i,j\} \notin \cP(D)$ there are different $k,l \in X$ with $\{i,k\} \in \pi$ and $\{j,l\} \in \pi$. Then
\[
|ij| + |kl| = |kl| \leq |ki| + |ij| + |jl| = |ki| + |jl|\ .
\]
Hence by replacing the pairs $\{i,k\}$, $\{j,l\}$ in $\pi$ with $\{i,j\}$, $\{k,l\}$ we obtain a new matching $\pi'$ with $m(\pi',D) \leq m(\pi,D)$. $m(\pi,D)$ is minimal among matchings, and therefore $\pi'$ too is a minimal matching, which witnesses the fact that $\{i,j\} \in \cP(D)$, a contradiction to the starting assumption.
\end{proof}

With this preparation we can prove the main result.

\begin{proof}[Proof of Theorem~\ref{mainthm}]
	Let $D$ be as in Lemma~\ref{minimallem}. Assume by contradiction that $(X,D)$ is not tree-like, i.e.\ renumbering the elements of $X$ if necessary,
	\begin{equation*}
		|13| + |24| > \max\{|12|+|34|,|14|+|23|\}\ .
	\end{equation*}
	By Lemma~\ref{preplem}, $\{1,3\},\{2,4\} \in \cP(D)$. This means that there are $\pi,\pi' \in \cM(X,D)$ with $\{1,3\} \in \pi$, $\{2,4\} \in \pi'$ and $m(\pi,D) = m(\pi',D) = m(X,D)$.
	We can write
	\begin{align*}
		\pi \:  & = \{\{1,3\},\{i_2,j_2\},\dots,\{i_n,j_n\}\}\ , \\
		\pi' & = \{\{2,4\},\{i'_2,j'_2\},\dots,\{i'_n,j'_n\}\}\ .
	\end{align*}
        We thus have
	\begin{equation}
	\label{sumformula}
	2m(X,D) = |13| + |24| + \sum_{m = 2}^n |i_mj_m| + |i'_mj'_m|\ .
	\end{equation}
	Every element of $X$ appears exactly twice in this sum because it is composed of two matchings.	Taking $F \defl \{\{i_m,j_m\},\{i'_m,j'_m\}\ |\ m=2,\ldots,n\}$ with repeated couples counted twice, consider the multigraphs (i.e.\ graphs with multiplicity) $(X,E)$, $(X,E_1)$ and $(X,E_2)$ given by
	\begin{eqnarray*}
		E \ &=& F \cup \{\{1,3\},\{2,4\}\}\ , \\
		E_1 &=& F \cup \{\{1,2\},\{3,4\}\}\ , \\
		E_2 &=& F \cup \{\{1,4\},\{2,3\}\}\ .
	\end{eqnarray*}
	By the remark above, every $x \in X$ has exactly two neighbors (counting multiplicities of edges) in $(X,E)$. Using this it is easy to check that the maximal connected subgraphs are cycles. Since $(X,E)$ is the union of two matchings, these cycles have even length, otherwise there would be two pairs in $\pi$ or $\pi'$ that have a point in common, which is not possible. Hence $(X,E)$ is the disjoint union of cycles of even length and with the same reasoning this is also true for the graphs $(X,E_1)$ and $(X,E_2)$. We consider separately the cases in which in $(X,E)$ the edges $\{1,3\}$ and $\{2,4\}$ belong to the same cycle or to different cycles.
	
	\begin{figure}[htbp!]
		\centering
		\resizebox{0.9\textwidth}{!}{\includegraphics[width=0.9\textwidth]{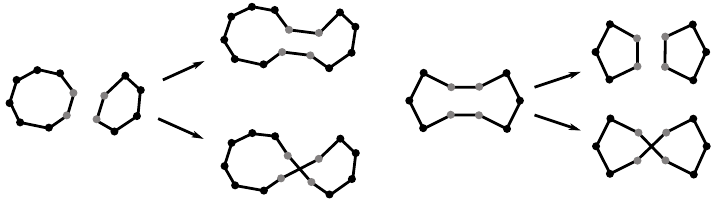}}
		\caption{We show what can happen going from $(X,E)$ to $(X, E_1), (X, E_2)$ in the two possible cases. The points $1,2,3,4$ are drawn in grey.}\label{figcycles}
	\end{figure}
	
	If $\{1,3\}$ and $\{2,4\}$ belong to different cycles of $(X,E)$ of lengths $2r, 2s$, then in both $(X, E_1)$ and $(X, E_2)$ the points $1,2,3,4$ belong to a single cycle of length $2r+2s$.
	
	If $\{1,3\}$ and $\{2,4\}$ belong to the same cycle $C$ of $(X,E)$ of length $2r$ then removing those edges from $C$ we are left with two paths $P, P'$ connecting either $1$ with $2$ and $3$ with $4$, or $1$ with $3$ and $2$ with $4$. $P$ and $P'$ belong to both $(X,E_1)$ and $(X, E_2)$ and have an added total length of $2r-2$. Moreover, in the first case, in $(X,E_1)$ the edges $\{\{1,2\},\{3,4\}\}\cup P\cup P'$ form a cycle of length $2r$ and in the second case the same is true for $(X,E_2)$.
	
	Therefore, in either case, one of $(X,E_1)$ or $(X,E_2)$ is a union of disjoint cycles of even lengths and by splitting (arbitrarily) each cycle into two sets of disjoint edges we obtain two matchings $\sigma$ and $\sigma'$. Comparing with \eqref{sumformula} leads to
	\begin{align*}
		m(\sigma,D) & + m(\sigma',D) \\
		& \leq \max\{|12|+|34|,|14|+|23|\} + \sum_{m = 2}^n |i_mj_m| + |i'_mj'_m| \\
		& < |13| + |24| + \sum_{m = 2}^n |i_mj_m| + |i'_mj'_m| = 2m(X,D)\ .
	\end{align*}
	But by the definition of the matching number, $m(\sigma,D) \geq m(X,D)$ and $m(\sigma',D) \geq m(X,D)$, a contradiction.
\end{proof}

Note that the tree furnished in the theorem is generally not unique, as shown in Figure~\ref{nonunique}.
\begin{figure}[htbp!]
	\centering
	\resizebox{0.8\textwidth}{!}{\includegraphics[width=0.8\textwidth]{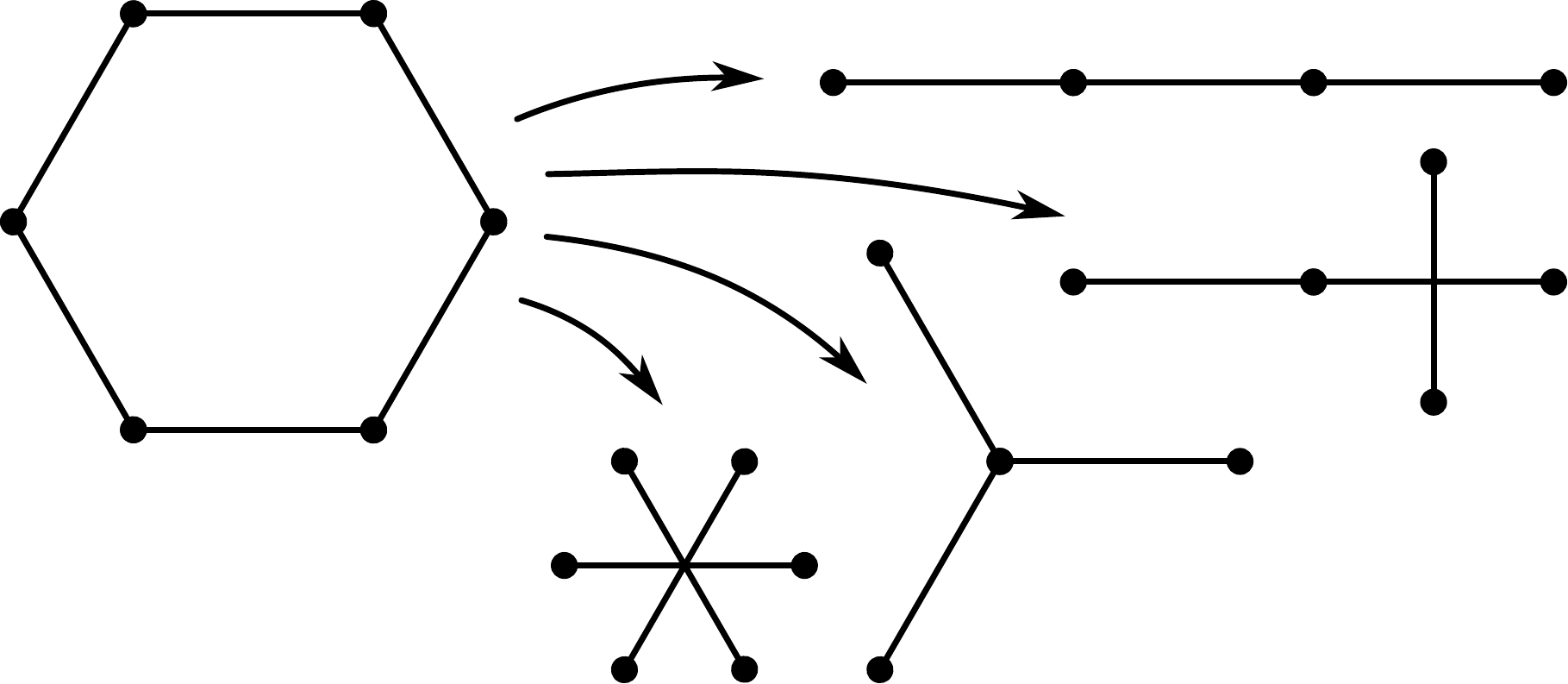}}
	\caption{Consider the cyclic graph with six vertices and the combinatorial (integer-valued) distance represented in the higher left corner. Shown on the right are the four possible metric trees of Theorem~\ref{mainthm}. If we perturb the combinatorial distance by adding further edges, then less and less of these trees stay admissible. The star-like tree with $2n$ points at distance $1/2$ from the center is admissible for the complete graph, and thus for any graph with $2n$ vertices and a matching of length $n$.}\label{nonunique}
\end{figure}

\subsection{Structure of the constructed tree}

The tree-like pseudometric $D$ of Theorem~\ref{mainthm} has some special features which we will discuss in this part. As noted in the introduction, there is a distance preserving map $f : (X,D) \to (T,d_T)$ into some complete metric tree, i.e.\ $d_T(f(x),f(x'))=D(x,x')$ for all $x,x' \in X$. Since $D \leq d$, this in particular defines a $1$-Lipschitz map $f : (X,d) \to (T,d_T)$. Complete metric trees are injective, see e.g.\ \cite[Lemma~2.1]{L} for a simple proof. As such, whenever the finite space $(X,d)$ is realized as a subspace of some metric space $(\tilde X,d)$, there is a $1$-Lipschitz extension $f : (\tilde X,d) \to (T,d_T)$. For a matching $\pi$ of $X$ and a map $f : X \to T$ into a tree we define the set
\[
A_\pi \defl \bigcup_{\{x,y\} \in \pi} [f(x),f(y)] \subset T \ .
\]
We will also use the set $A_X \subset T$ as defined in \eqref{defamxd}. For points $u,v,w$ in a tree $T$ we denote by $c(u,v,w)$ the intersection of the three open arcs $]u,v[$, $]v,w[$ and $]w,u[$. For a map $g : Y \to T$ defined on a set $Y$ we denote
\begin{equation}\label{vfx}
V_g(Y) \defl \{c(g(x),g(y),g(z))\ |\ x,y,z\in Y\} \subset T\ .
\end{equation}
This set equals the set of internal vertices of the subtree in $T$ spanned by $g(Y)$. This is straightforward to see once we note that $c(u,v,w)$ is empty if one of $u,v,w$ is included in the closed segment formed by the others, and is the internal vertex of the tripod spanned by $u,v,w$ otherwise.

We say that $(T,d_T)$ is a \emph{minimal tree} for the pseudometric space $(X,d)$ if there is a $1$-Lipschitz map $f : (X,d) \to (T,d_T)$ such that $D = f^* d_T$ is minimal as in Lemma~\ref{minimallem} and $T$ is spanned by $f(X)$.

\begin{Prop}
	\label{mainthmbis}
	For any pseudometric $d$ on a set $X$ with $\# X = 2n$, there is a metric tree $(T,d_T)$ and a $1$-Lipschitz map $f : X \to T$ such that $m(X,f^*d_T) = m(X,d)$. Assuming such a map, let $\pi \in \cM(X,f^*d_T)$ (note that $\cM(X,d) \subset \cM(X,f^*d_T)$). Then we have the following properties:
	\begin{enumerate}
		\item For a pair $\{x,y\}$ that appears in a minimal matching of $(X,d)$, there holds $d_T(f(x),f(y)) = d(x,y)$. Assume further $(\tilde X,d)$ contains $X$ as a subset, $f : \tilde X \to T$ is a $1$-Lipschitz extension and $[x,y]$ is a geodesic segment connecting $x$ with $y$ in $\tilde X$, then the restriction $f : [x,y] \to [f(x),f(y)]$ is an isometry.  \label{1pa}
		\item For different matches $\{x,y\}$ and $\{x',y'\}$ in $\pi$, the arcs $[f(x),f(y)]$ and $[f(x'),f(y')]$ have at most one common point. Hence $\cH^1(A_{\pi}) = m(X,d)$. \label{2pa}
		\item $A_\pi \subset A_{\pi'}$ for any other matching $\pi'$ of $X$. In particular $A_\pi = A_X$ and $A_\pi = T$ if $(T,d_T)$ is a minimal tree. \label{3pa}
		\item For all points $p \in A_\pi \setminus V_f(X)$ there are components $C$ of $A_\pi \setminus \{p\}$ for which $\#\{x \in X\ |\ f(x) \in C\}$ is odd. \label{4pa}
	\end{enumerate}
\end{Prop}

\begin{proof}
	As before we abbreviate $|xy| = f^* d_T$. Let $\pi'$ be a minimal matching for $(X,d)$. By assumption, $m(\pi,d) = m(X,f^*d_T)$ and hence $d(x,y)=d_T(f(x),f(y))$ for $\{x,y\} \in \pi$, otherwise we would have $m(X,f^*d_T) < m(X,d)$. Let $f : \tilde X \to T$ be any $1$-Lipschitz extension. If $\{x,y\} \in \pi$ and $[x,y]$ is a geodesic in $\tilde X$ connecting $x$ with $y$, then $d_T(f(x'),f(y')) = d(x',y')$ for any two points $x',y' \in [x,y]$ since $f$ is $1$-Lipschitz and $d_T(f(x),f(y)) = d(x,y)$. This shows \eqref{1pa}.
	
	Let $\{x,y\}$ and $\{x',y'\}$ be two different pairs in $\pi$. Indeed, if the intersection $[f(x),f(y)] \cap [f(x'),f(y')]$ would contain more than one point, it would contain a nontrivial arc. But then
	\[
	\min\{|xx'| + |yy'|, |xy'| + |x'y|\} < |xy| + |x'y'|\ ,
	\]
	and by replacing the pairs $\{p,q\}$, $\{p',q'\}$ in $\pi$ with $\{p,p'\}$, $\{q,q'\}$ or $\{p,q'\}$, $\{p',q\}$ we obtain a new matching $\pi'$ with $m(\pi',f^* d_T) < m(\pi,f^* d_T)$, which is not possible. This proves \eqref{2pa}. Moreover it implies $\cH^1(A_{\pi}) = m(X,d)$.

	To prove \eqref{3pa} it suffices to prove that $A_\pi \subset A_{\pi'}$ for any matching $\pi'$ of $X$. This then shows that $A_\pi = A_X$. In case $T$ is a minimal tree, then any pair $\{x,y\}$ is contained in a minimal matching of $(X,f^*d_T)$ by Lemma~\ref{preplem}. Since $T$ is also spanned by $f(X)$ this shows that $A_\pi = T$. Assume by contradiction that $A_{\pi} \setminus A_{\pi'}$ is nonempty. Let $T' \subset T$ be the subtree spanned by $f(X)$. Since both $A_\pi$ and $A_{\pi'}$ are finite unions of compact arcs, there is a nontrivial arc $]a,b[\ \subset A_{\pi} \setminus A_{\pi'}$. Since $T'$ is a finite tree, we can assume that $]a,b[$ does not intersect the set $V_f(X)$ defined in \eqref{vfx}. Hence, $T' \setminus ]a,b[$ consists of exactly two components. Denote by $B$ one of these components and let $Y \subset X$ be those points that get mapped into $B$ by $f$. Since $]a,b[ \cap A_{\pi'}$ is empty, $Y$ contains an even number of points. Otherwise there would be a matching $\{x,y\} \in \pi'$ with $]a,b[\ \subset [f(x),f(y)]$. Since $]a,b[$ doesn't contain any vertices of $T'$ and $]a,b[ \subset A_\pi$, \eqref{2pa} implies that there is exactly one matching $\{x,y\} \in \pi$ with $]a,b[\ \subset [f(x),f(y)]$. Hence $Y$ is odd, which gives a contradiction. Hence, $A_{\pi} \subset A_{\pi'}$ and with the same argument also $A_{\pi'} \subset A_{\pi}$. This establishes \eqref{3pa}.
	
	Let $A$ be a connected component of $A_{\pi}$. By \eqref{2pa}, $\# \{x \in X\ | \ f(x) \in A\}$ is even and for $p \in A \setminus V_f(X)$ any of the two connected components $C$ of $A \setminus \{p\}$ the number $\# \{x \in X\ | \ f(x) \in C\}$ is odd. This proves \eqref{4pa}.
	
\end{proof}

\begin{figure}[htbp!]
	\centering
	\resizebox{0.8\textwidth}{!}{\includegraphics[width=0.8\textwidth]{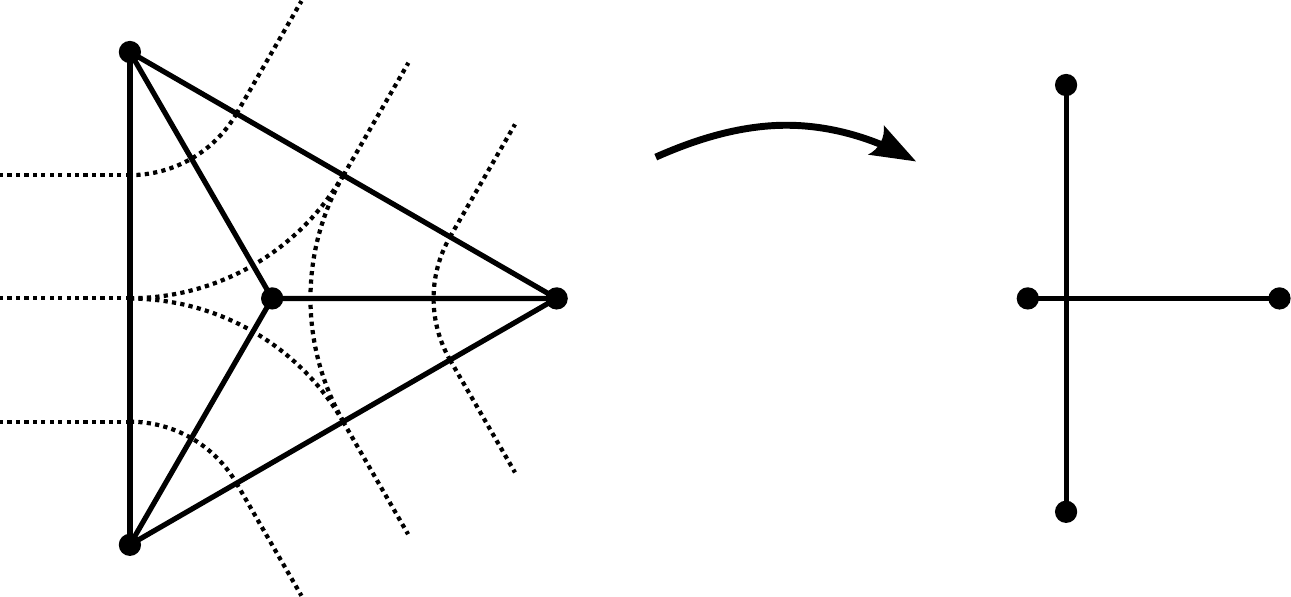}}
	\caption{Illustrated is a $1$-Lipschitz map $f : \R^2 \to T$ as in Proposition~\ref{mainthmbis} corresponding to the set $X$ displayed by the four dots and mutual geodesics by thick lines on the left. The dotted lines indicate some possible level sets.}\label{figcollapse}
\end{figure}

\subsection{Proof of Theorem~\ref{mainthmcal}}
\label{proofofmainthmcal}

Recall that given a metric space $\tilde X$, the notation $H_1^{(\Lip)}(\tilde X)$ represents the first singular (Lipschitz) homology group of $\tilde X$. In particular the condition $H_1^{\Lip}(\tilde X)=0$ means that in $\tilde X$ all Lipschitz cycles are boundaries of Lipschitz $2$-chains. Similarly, $\pi_1^{(\Lip)}(\tilde X) = 0$ means that for any (Lipschitz) map $\gamma : S^1 \to \tilde X$ there is a (Lipschitz) extension $\Gamma : \B^2(0,1) \to \tilde X$ with $\Gamma|_{S^1} = \gamma$. Apart from the construction of the tree in \cite{Z} we will make use of the following lemma, which is probably well known. 

\begin{Lem}[\cite{Z}, Lemma~3.4]
	\label{disconnectlem}
Let $\tilde X$ be a connected and locally (Lipschitz) path-connected space with $H_1^{(\Lip)}(\tilde X) = 0$. Assume that $C \subset \tilde X$ is a closed set that disconnects two points $x$ and $x'$ in $\tilde X$. Then there is a connected component of $C$ that disconnects $x$ and $x'$.
\end{Lem}

Additionally, the following facts will be used in the proof of Theorem~\ref{mainthmcal}. Let $C = \curr \Gamma \in \cR_1(Y,\Z_2)$ where $\Gamma\subset Y$ is an $\cH^1$-rectifiable subset of $Y$, and let $f : Y \to Z$ and $g : Z \to \R$ be Lipschitz maps. As in \cite[Subsection~3.5]{depauwhardt}, the push-forward $f_\# C \in \cR_1(Z,\Z_2)$ is defined by $\sum_{i=1}^\infty \curr{(f\circ \gamma_i)(K_i)}$, where each $f\circ \gamma_i : K_i \to Z$ is bi-Lipschitz, $K_i \subset \R$ is compact, the images $\gamma_i(K_i) \subset \Gamma$ are pairwise disjoint and $\cH^1(f(\Gamma \setminus \cup_{i=1}^\infty \gamma_i(K_i))) = 0$. Then
\begin{equation}
\label{pushforwardestimate}
C(d(g \circ f)) \geq \sum_i \int_{K_i} |(g \circ f \circ \gamma_i)'(t)| \, d\cH^1(t) = (f_\# C)(dg) \ .
\end{equation}

Let $X \subset \tilde X$ be a set consisting of an even number of points in a geodesic metric space $(\tilde X, d)$, then as noted in the beginning of Subsection~\ref{platmodp},
\begin{equation}
\label{massminimizer}
\F_{\Z_2}(\curr X) = m(X,d) \ ,
\end{equation}
and the minimum is achieved if $C=\sum_{i=1}^n\curr{x_i, y_i}$ where $[x_i,y_i]$ are geodesic segments and $\{\{x_i,y_i\},1\le i\le n\}$ is a minimal matching for $(X,d)$.

Let $\gamma : [0,1] \to T$ be a Lipschitz curve into some metric tree $(T,d_T)$. Then
\begin{equation}
\label{chainsinatree}
\gamma_\# \curr{0,1} = \curr{\gamma(0),\gamma(1)} \text{ in } \cR_1(T, \Z_2) \ .
\end{equation}
This is an immediate consequence of the fact that $\gamma_\#\curr{S^1} = 0$ for every closed Lipschitz curve $\gamma: S^1 \to T$. Any such $\gamma$ has a Lipschitz extension $g : \B^2(0,1) \to T$ with $\op{im}(g) = \op{im}(\gamma)$ (for example, let $q \in \op{im}(\gamma)$ and define $g(te) \defl [q,\gamma(e)](t)$). This implies $\cH^2(\op{im}(g)) = \cH^2(\op{im}(\gamma)) = 0$ and hence $\gamma_\#\curr{S^1} = \partial (g_\#\curr{\B^2(0,1)}) = 0$.

\begin{proof}[Proof of Theorem~\ref{mainthmcal}:]
$\mathbf{(1) \Rightarrow (2)}$:
Assume that $\# X = 2n$ and let $f$ and $\rho$ be as in (1) and $C \in \cR_1(\tilde X,\Z_2)$ with $\partial C = \curr X$. Assume first that $C = \sum_{i = 1}^n {\gamma_i}_\# \curr{0,1}$, where $\gamma_i : [0,1] \to X$ are Lipschitz curves with $\gamma_i(t) \in X$ for all $i$ and $t = 0,1$. Then $C_T \defl f_\# C$ is a $1$-chain in $T$ with $\partial (C_T) = f_\#\curr X$. From \eqref{chainsinatree} it follows that
\[
C_T = \sum_{i=1}^n \curr{f(\gamma_i(0)),f(\gamma_i(1))} \ .
\]
Let $\pi$ be a minimal matching for $m(X,f^*d_T)$. From Proposition~\ref{mainthmbis}(4) it follows that for all $p \in A_\pi \setminus V_f(X)$, there is a component $A_{odd}$ of $A_\pi \setminus \{p\}$ such that $\#\{x \in X\ |\ f(x) \in A_{{odd}}\}$ is odd. Hence, $p \in \spt(C_T)$ and therefore $A_\pi \subset \spt(C_T)$. Since $\rho$ is an orientation modulo $2$ for $A_X$, this shows together with Proposition~\ref{mainthmbis}(2), \eqref{pushforwardestimate} and \eqref{massminimizer} that
\[
\F_{\Z_2}(\curr X) = m(X,d) = \cH^1(A_\pi) \leq C_T(d\rho) \leq C(d(\rho \circ f)) \ .
\]
This shows (2) for $C$ and by a simple argument for any Lipschitz chain. The general case follows by approximation.

$\mathbf{(1) \Rightarrow (3)}$:
Let $f$ and $\rho$ be as in (1) and let $\pi$ be a minimal matching of $(X,d)$, i.e.\ $m(\pi,d) = m(X,d)$. Let $A \subset T \setminus f(X)$ be some set and $\{x,y\} \in \pi$. If $x$ and $y$ are in the same component $C$ of $\tilde X \setminus f^{-1}(A)$, then $f(C)$ is a connected set containing $f(x)$ and $f(y)$ but does not intersect $A$. Since a set in $T$ is connected if and only if it is arcwise connected, $A$ doesn't intersect $[f(x),f(y)]$. On the other side, if $f(x)$ and $f(y)$ are in the same component of $T \setminus A$, then $A$ doesn't intersect $[f(x),f(y)]$ and since $f : [x,y] \to [f(x),f(y)]$ is an isometry by Proposition~\ref{mainthmbis}(1), $[x,y]$ does not intersect $f^{-1}(A)$. Hence $x$ and $y$ are in the same connected component of $\tilde X \setminus f^{-1}(A)$. This shows that for a set $A \subset T \setminus f(X)$,
\begin{equation}
\label{disconnection}
\begin{array}{c} A \text{ disconnects } f(x) \text{ and } f(y) \text{ in } T \text{ if and} \\ \text{only if } f^{-1}(A) \text{ disconnects } x \text{ and } y \text{ in } \tilde X. \end{array}
\end{equation}
Now assume further that $A \subset T \setminus V_f(X)$ is closed and connected. Then by Proposition~\ref{mainthmbis}(2), $A$ intersects at most one arc $[f(x),f(y)]$ for $\{x,y\} \in \pi$. If $A \cap [f(x),f(y)]$ is nonempty for some $\{x,y\} \in \pi$, \eqref{disconnection} shows that $f^{-1}(A \cap A_\pi)$ disconnects $x$ and $y$ in $\tilde X$ while all other matches in $\pi$ are not disconnected by $f^{-1}(A)$. From Lemma~\ref{disconnectlem} it follows that there is at least one connected component of $f^{-1}(A \cap A_\pi)$ that disconnects $x$ and $y$ and hence
\begin{align} 
	\nonumber
\op{Cut}_{\mathbb Z_2}(f^{-1}(A),X) & = \op{Cut}_{\mathbb Z_2}(f^{-1}(A \cap A_\pi),X) \\
\label{cutfbound}
 & = 
\left\{
\begin{array}{ll}
n_A \geq 1  & \mbox{if } A \cap A_\pi \neq \emptyset \ , \\
0 & \mbox{if } A \cap A_\pi = \emptyset \ .
\end{array}
\right.
\end{align}
This in particular holds for $A$ consisting of a single point outside $V_f(X)$. From the definition in \eqref{vfx} we see that $V_f(X)$ is a finite set, and by Proposition~\ref{mainthmbis}(2) we have $\cH^1(A_\pi) = m(X,d)$, therefore
\begin{equation} \label{levbound}
m(X,d) = \cH^1(A_\pi) \leq \int_{A_\pi} \op{Cut}_{\mathbb Z_2}(f=q,X) \, d\mathcal H^1(q) \ .
\end{equation}
From the area formula and since $\rho$ is $1$-Lipschitz it follows,
\[
\int_{\R} \#\bigl(\rho^{-1}(t) \cap A_\pi\bigr) \, dt = \int_{A_\pi} J(\rho|_{A_\pi})(q) \, d\cH^1(q) \leq \cH^1(A_\pi) < \infty \ .
\]
This shows that $\#(\rho^{-1}(t) \cap A_\pi)$ is finite for almost every $t$. Fix some $t \notin \rho(V_f(X))$ and let $\mathcal A_t$ be the collection of connected components of $\rho^{-1}(t)$ in $T$. Since any $A \in \mathcal A_t$ intersects $A_\pi$ in at most one point (otherwise $J(\rho|_{A_\pi})$ would vanish on some interval, contradicting the hypothesis that $\rho$ is an orientation modulo $2$, as $A_\pi\subset A_X$) we obtain by \eqref{cutfbound},
\begin{align}
	\nonumber
\op{Cut}_{\mathbb Z_2}(\varphi = t,X) & = \sum_{A \in \mathcal A_t} \op{Cut}_{\mathbb Z_2}(f^{-1}(A), X) \\
\label{cutphif}
 & = \sum_{q \in \rho^{-1}(t) \cap A_\pi} \op{Cut}_{\mathbb Z_2}(f = q, X) \ .
\end{align}
Applying the area formula with $g(q) \defl \op{Cut}_{\mathbb Z_2}(f=q,X)$ together with \eqref{cutphif} we get
\begin{align*}
\int_{A_\pi} J(\rho|_{A_\pi})(q) g(q) \, d\mathcal H^1(q) & = \int_{\R} \sum_{q\in\rho^{-1}(t) \cap A_\pi} g(q) \, dt \\
 & = \op{lev}_{\mathbb Z_2}(\varphi,X)\ .
\end{align*}
By the definition of $\rho$ there holds $J(\rho|_{A_\pi})(q) = 1$ for $\cH^1$-a.e.\ $q \in A_\pi$. With \eqref{massminimizer} and \eqref{levbound} we conclude that
\begin{equation}\label{maz}
\F_{\mathbb Z_2}(\curr X) = m(X,d) \leq \op{lev}_{\mathbb Z_2}(\varphi,X) \leq \Lev_{\mathbb Z_2}(X)\ .
\end{equation}
Next we show that $\Lev_{\mathbb Z_2}(X) \leq \F_{\mathbb Z_2}(\curr X)$ holds. Indeed, let $\Gamma$ be a geodesic segment that connects $x$ with $y$ in $\tilde X$ and $g : \tilde X \to \R$ be $1$-Lipschitz. Then via the area formula there holds
\begin{equation}
\label{jacobianbound}
\int_{\mathbb R} \# (g^{-1}(t) \cap \Gamma) \,dt = \int_\Gamma J(g|_{\Gamma})(s) \,d\mathcal H^1(s) \leq \cH^1(\Gamma) = d(x,y) \ .
\end{equation}
Clearly, $\# (g^{-1}(t) \cap \Gamma)$ is an upper bound on the number of components of $g^{-1}(t)$ that separate $x$ from $y$ in $\tilde X$. Hence, $\op{lev}_{\Z_2}(g,\{x,y\}) \leq d(x,y)$ and summing over all pairs of $\pi$ we get
\begin{align*}
\op{lev}_{\Z_2}(g,X) & \leq \sum_{\{x,y\} \in \pi} \op{lev}_{\Z_2}(g,\{x,y\}) \leq m(\pi,d) = m(X,d) \\
 & = \F_{\mathbb Z_2}(\curr X) \ .
\end{align*}
This concludes the proof of this part and since maps $f$ and $\rho$ as in (1) exist by Proposition~\ref{mainthmbis}, this also shows that in case $H_1(\tilde X) = 0$ or $H_1^{\Lip}(\tilde X) = 0$ we have
\begin{equation}
\label{filllevequal}
\F_{\mathbb Z_2}(\curr X) = m(X,d) = \Lev_{\mathbb Z_2}(X) \ .
\end{equation}

\textbf{$\mathbf{(3) \Rightarrow (1)}$: }
Let $\varphi : \tilde X \to \R$ be as in (3). By \eqref{filllevequal} just above, we know that $\op{lev}_{\mathbb Z_2}(\varphi,X) = m(X,d)$. As in the proof of \cite[Theorem~1]{Z} consider the set $T = \tilde X /_\sim$, where $x \sim x'$ if $D(x,x') = 0$ with the pseudo distance $D$ on $\tilde X$ given by
\begin{equation}
\label{defD}
D(x,x') \defl \inf \bigl\{\diam(\varphi(C))\ |\ x,x' \in C \text{ and } C \subset \tilde X \text{ is connected} \bigr\} \ .
\end{equation}
Let $f : \tilde X \to T$ be the quotient map and $\rho : T \to \R$ the map for which $\varphi = \rho \circ f$ holds. It is shown in \cite[Lemma~3.1]{Z} that $(T,D)$ is a metric space and both $f$ and $\rho$ are $1$-Lipschitz. Moreover, it follows from \cite[Proposition~3.8]{Z} that $(T,D)$ is a (topological) tree. Let $d_T$ be the intrinsic metric induced by $(T,D)$, i.e.\ $d_T(p,p')$ is the minimal length of curves in $(T,D)$ connecting $p$ with $p'$ in $T$. Because $\tilde X$ is geodesic and $f : (\tilde X,d) \to (T,D)$ is onto and $1$-Lipschitz, we immediately get that $f : (\tilde X,d) \to (T,d_T)$ is also $1$-Lipschitz. By construction, $d_T \geq D$ and hence $\rho : (T,d_T) \to \R$ is also $1$-Lipschitz. Let $\pi$ be a minimal matching of $(X,d)$. From the area formula it follows for a geodesic segment $[x,y]$ connecting $x$ with $y$ in $\tilde X$ as in \eqref{jacobianbound},
\begin{align*}
\op{lev}(\varphi = t,\{x,y\}) & = \int_{\mathbb R} \op{Cut}_{\Z_2}(\varphi = t,\{x,y\}) \,dt \\
 & \leq \int_{\mathbb R} \# \bigl(\varphi^{-1}(t) \cap [x,y]\bigr) \,dt \\
 & \leq \int_{[x,y]} J(\varphi|_{[x,y]})(s) \,d\mathcal H^1(s) \\
 & \leq d(x,y) \ .
\end{align*}
Hence,
\begin{align*}
m(X,d) & = \op{lev}_{\mathbb Z_2}(\varphi,X) \leq \sum_{\{x,y\} \in \pi} \op{lev}(\varphi = t,\{x,y\}) \\
 & \leq m(\pi,d) = m(X,d) \ .
\end{align*}
This shows that $\op{lev}(\varphi = t,\{x,y\}) = d(x,y)$ for all $\{x,y\} \in \pi$ and further there is a $\cH^1$-measurable set $G \subset [x,y]$ with
\begin{equation}
\label{cutequalpreimage}
\left\{
\begin{array}{l}
\cH^1([x,y] \setminus G) = 0 \ , \\
J(\varphi|_{[x,y]})(s) = 1 \text{ for all } s \in G \ , \\
\# (\varphi^{-1}(t) \cap [x,y]) = \op{Cut}_{\Z_2}(\varphi = t,\{x,y\}) < \infty \\
\text{for all } t \in \varphi(G) \ .
\end{array}
\right.
\end{equation}
This means that for $t \in \varphi(G)$ every point $s$ in the finite set $\varphi^{-1}(t) \cap [x,y]$ comes from a $\Z_2$-cut component of $\varphi^{-1}(t)$ and $J(\varphi|_{[x,y]})(s) = 1$. From the construction of $T$ it is clear that every connected component $c$ of $\varphi^{-1}(t)$ satisfies $f(c) = p$ for some $p \in T$. Now assume by contradiction that there are two different points $x < s_1 < s_2 < y$ in $G$ with $\varphi(s_1) = \varphi(s_2) = t$ and $f(s_1) = f(s_2) = p$. By \eqref{cutequalpreimage} there are components $c_1$ and $c_2$ of $\varphi^{-1}(t)$ that disconnect $x$ and $y$ and $c_i \cap [x,y] = s_i$, $i=1,2$. Since $J(\varphi|_{[x,y]})(s_1) = 1$, there is some $s_3 \in G \cap ]s_1,s_2[$ close to $s_1$ with $\varphi(s_3) = t' \neq t$. Let $c_3$ be the corresponding component of $\varphi^{-1}(t')$ with $c_3 \cap [x,y] = s_3$. From the definition of $D$ in \eqref{defD} it follows that
\[
d_T(p,p) \geq D(p,p) = D(s_1,s_2) \geq |t-t'| \ ,
\]
a contradiction. To see the last estimate, let $C$ be a connected set in $\tilde X$ that contains $s_1$ and $s_2$. Since $c_3$ disconnects $x$ and $y$ in $\tilde X$, $C \cap c_3$ is nonempty and hence $\diam(\varphi(C)) \geq |t-t'|$. So the restriction $f|_G$ is injective and satisfies $J(f|_{[x,y]})(s) = 1$ for $\cH^1$-a.e.\ $s \in [x,y]$. The latter is implied by the fact that $f$ and $\rho$ are $1$-Lipschitz and using the chain rule leads to
\begin{equation}
\label{jacobians}
1 = J(\varphi|_{[x,y]})(s) = J(\rho|_{f([x,y])})(f(s)) J(f|_{[x,y]})(s) \ ,
\end{equation}
which holds for a.e.\ $s \in [x,y]$. Therefore $f|_{[x,y]} : [x,y] \to [f(x),f(y)]$ is an isometry and in particular $d_T(f(x),f(y)) = d(x,y)$ for all $\{x,y\} \in \pi$. Hence we obtain $m(\pi,d) = m(\pi,f^*d_T)$ and
\begin{equation*}
\op{Cut}_{\mathbb Z_2}(p,\{x,y\}) = 
\left\{
\begin{array}{l}
1 \text{ if } p \in [f(x),f(y)] \cap f(G) \ , \\
0 \text{ if } p \notin [f(x),f(y)] \ .
\end{array}
\right.
\end{equation*}
This implies
\begin{align*}
\int_{A_\pi} \op{Cut}_{\mathbb Z_2}(p,X) \, d\cH^1(p) & \leq \sum_{\{x,y\} \in \pi} \int_{[f(x),f(y)]} \op{Cut}_{\mathbb Z_2}(p,\{x,y\}) \, d\cH^1(p) \\
 & \leq \cH^1(A_\pi) \leq m(X,d) \ .
\end{align*}
Since every component $c$ of $\varphi^{-1}(t)$ maps to some single point in $T$ there holds
\[
\op{Cut}_{\mathbb Z_2}(\varphi = t,X) = \sum_{p \in \rho^{-1}(t)} \op{Cut}_{\mathbb Z_2}(p,X) \ ,
\]
for all $t \in \R$. Since $\rho$ is $1$-Lipschitz it follows from the area formula, \eqref{filllevequal} and the two equations above,
\begin{align*}
m(X,d) & = \int_{\R} \op{Cut}_{\mathbb Z_2}(\varphi = t,X) \, dt \leq \int_{A_\pi} \op{Cut}_{\mathbb Z_2}(p,X) \, d\cH^1(p) \\
 &  \leq \cH^1(A_\pi) \leq m(X,d) \ .
\end{align*}
Hence $\cH^1(A_\pi) = m(X,d)$ and as in Proposition~\ref{mainthmbis}(2), for two different pairs $\{x,y\},\{x',y'\} \in \pi$ the intersection $[f(x),f(y)] \cap [f(x'),f(y')]$ contains at most one point. Assume by contradiction that there is a matching $\pi'$ of $X$ with $m(\pi',f^*d_T) < m(\pi,f^*d_T)$. If we consider $A_{\pi'} = \cup_{\{x,y\} \in \pi'} [f(x),f(y)]$, this assumption implies $\cH^1(A_{\pi'}) < \cH^1(A_{\pi})$. Then the same argument as in the proof of Proposition~\ref{mainthmbis}(3) gives a contradiction and hence $m(X,f^*d_T) = m(X,d)$. From \eqref{jacobians} it follows directly that $\rho$ is an orientation modulo $2$ for $A_\pi$, which equals $A_X$ by Proposition~\ref{mainthmbis}(3).

\textbf{$\mathbf{(2) \Rightarrow (1)}$: }
Let $\varphi : \tilde X \to \R$ be as in (2). As in \cite{WY} consider the pseudo distance $d_\varphi$ on $\tilde X$ defined by
\[
d_\varphi(x,y) \defl \inf\bigl\{\op{length}(\varphi \circ \gamma)\ |\ \gamma \text{ is a Lipschitz curve connecting } x \text{ with } y \bigr\} \ .
\]

Let $T = \tilde X/_\sim$ with $x \sim x'$ if $d_\varphi(x,y) = 0$. It is stated in \cite[Theorem~5]{WY}, respectively in the proof thereof, that $(T,d_\varphi)$ is a metric tree and there are $1$-Lipschitz maps $f : \tilde X \to T$ and $\rho : T \to \R$ with $\varphi = \rho \circ f$.

Let $\pi$ be a minimal matching for $(X,d)$. For any $\{x,y\} \in \pi$ choose a geodesic segment $[x,y]$ in $\tilde X$. Assume by contradiction that $f$ is not injective on $[x,y]$. Then there are points $x \leq v < w \leq y$ on $[x,y]$ with $f(v) = f(w)$. By the definition of $d_\varphi$, there is a sequence of Lipschitz curves $\gamma_n : [0,1] \to \tilde X$ connecting $v$ with $w$ such that
\begin{align*}
0 & = d_\varphi(f(v),f(w)) = d_\varphi(v,w) = \lim_{n \to \infty} \op{length}(\varphi \circ \gamma_n) \\
 & = \lim_{n \to \infty} \int_0^1|(\varphi \circ \gamma_n)'(s)| \, ds \ .
\end{align*}
Replacing the $\gamma_n$ by injective curves if necessary we get
\[
\lim_{n\to\infty}\gamma_{n\#}\curr{0,1}(d\varphi) = \lim_{n \to \infty} \int_0^1|(\varphi \circ \gamma_n)'(s)| \, ds = 0 \ .
\]
If we set $C_n \defl \curr{x,v} + \gamma_{n\#}\curr{0,1}(d\varphi) + \curr{w,y}$, then $\partial C_n = \curr x + \curr y$ and for $n$ large
\begin{align*}
	C_n(d\varphi) & \leq \curr{x,v}(d\varphi) + \gamma_{n\#}\curr{0,1}(d\varphi) + \curr{w,y}(d\varphi) \\
	& \leq d(x,v) + d(w,x) + \gamma_{n\#}\curr{0,1}(d\varphi) < d(x,y) \ .
\end{align*}
This contradicts our assumption on $\varphi$. Namely, from \eqref{massminimizer} it follows that for all $n$, $m(X,d) = \F_{\Z_2}(\curr X) \leq C_n(d\varphi)$.

Therefore, $f$ is injective on $[x,y]$. By the assumption on $\varphi$ there holds $J(\varphi|_{[x,y]})(p) = 1$ for $\cH^1$-a.e.\ $p \in [x,y]$ and since both $f$ and $\rho$ are $1$-Lipschitz, the chain rule \eqref{jacobians} implies that $J(f|_{[x,y]})(p) = 1$ for $\cH^1$-a.e.\ $p \in [x,y]$. Hence the restriction of $f$ to $[x,y]$ is an isometry. This is true for any $\{x,y\} \in \pi$, thus $m(\pi,d) = m(\pi,f^*d_\varphi)$. Now assume that $[f(x),f(y)] \cap [f(x'),f(y')]$ is nonempty for some different $\{x,y\},\{x',y'\} \in \pi$. If this intersection would contain an arc $[a,b]$ for different $a,b \in T$, then there are points $x \leq v < w \leq y$ in $[x,y]$ and points $x' \leq v' < w' \leq y'$ in $[x',y']$ with $f(v) = f(v') = a$ and $f(w) = f(w') = b$. Connecting $x$ with $x'$ via Lipschitz curves from $v$ to $v'$ and $y$ with $y'$ via Lipschitz curves from $w$ with $w'$ as above, we get a contradiction to our starting assumption on $\varphi$.

Combining these observation we get the equalities $\cH^1(A_\pi) = m(\pi,d) = m(\pi,f^*d_\varphi)$ for any minimal matching $\pi$ for $(X,d)$. Assume by contradiction that there is some matching $\pi'$ of $X$ that satisfies $m(\pi',f^*d_\varphi) < m(\pi,f^*d_\varphi)$. Then $\cH^1(A_{\pi'}) < \cH^1(A_\pi)$ and as in the proof of Proposition~\ref{mainthmbis}(3) we get a contradiction. We have already established that $J(f|_{[x,y]})(p) = 1$ for $\cH^1$-a.e.\ $p \in [x,y]$ in case $\{x,y\} \in \pi$. Again with the chain rule \eqref{jacobians} it follows directly that $\rho$ is an orientation modulo $2$ for $A_\pi$, which equals $A_X$ by Proposition~\ref{mainthmbis}(3).
\end{proof}

\section{Two generalizations of matching numbers}

\subsection{Matching number and dimension for metric spaces}
\label{packing}
For a metric space $(X,d)$, an even number $k \in \N$ and $\epsilon > 0$ define the \emph{matching numbers}
\begin{align*}
	m_{k}(X,d) & \defl \sup \bigl\{m(X',d)\ |\ X' \subset X, \# X' = k \bigr\}\ , \\
	m'_{\epsilon}(X,d) & \defl \sup \bigl\{m(X',d)\ |\ X' \subset X \text{ is } \epsilon \text{-separated} \bigr\}\ .
\end{align*}
Here, $X'$ is $\epsilon$-separated if $d(x,x')\geq \epsilon$ for different $x,x' \in X'$. To make the arguments simpler, we allow for members of $X'$ in the definition of $m_{k}(X,d)$ to appear more than once, i.e.\ $X'$ is a multiset. This way we also don't run into the problem of taking a supremum of the empty set. This can happen in the definition of $m'_{\epsilon}(X,d)$ if $\epsilon > \diam(X,d)$. In this case we set $\sup \emptyset \defl 0$. Obviously, $m_{k}(X,d) = m'_{\epsilon}(X,d) = \infty$ if $X$ is not bounded. Here are some easy observations about these numbers, the proofs of which are elementary.

\begin{Lem}
\label{basicproperties}
The following properties for the matching numbers hold,
\begin{enumerate}
	\item If $A \subset X$, then $m_{k}(A,d) \leq m_{k}(X,d)$ and $m'_{\epsilon}(A,d) \leq m'_{\epsilon}(X,d)$.
	\item For even numbers $k \leq k'$ and reals $\epsilon \leq \epsilon'$,
		\begin{align*}
			m_{k}(X,d) & \leq m_{k'}(X,d) \leq \diam(X)\frac{k}{2} \ , \\
			m'_{\epsilon'}(X,d) & \leq m'_{\epsilon}(X,d) \ .
		\end{align*}
	\item For any Lipschitz map $\varphi : (X,d_X) \to (Y,d_Y)$,
		\[
		m_{k}(\varphi(X),d_Y) \leq \Lip(\varphi) m_{k}(X,d_X) \ .
		\]
	\item If $\varphi : (X,d_X) \to (Y,d_Y)$ is bi-Lipschitz,
		\[
		\Lip(\varphi^{-1})^{-1} m'_{\Lip(\varphi^{-1})\epsilon}(X,d_X) \leq m'_{\epsilon}(Y,d_Y) \leq \Lip(\varphi) m'_{\Lip(\varphi)^{-1}\epsilon}(X,d_X)\ .
		\]
	\end{enumerate} 
\end{Lem}

Depending on some geometric conditions of a metric space we give some bounds to these matching numbers.

\begin{Prop}\label{axiomaticmatching}
Let $(X,d)$ be a compact metric space and $n \geq 1$. Assume that there are constants $0 < c_1 < C_1$ such that for every $0 < \epsilon < \diam(X)$,
\begin{equation}\label{firstinequality}
c_1 \epsilon^{-n} < \sup\{\# X'\ |\ X' \subset X \text{ is } \epsilon\text{-separated, } \# X' \text{ even}\} \leq C_1 \epsilon^{-n} \ .
\end{equation}
Then, there is a constant $c > 0$ such that for all $0 < \epsilon < \diam(X)$ and all even numbers $k$,
\begin{equation}
\label{lowerbound}
m_k(X,d) \geq c k^\frac{n-1}{n} \ , \quad \text{ and } \quad m'_{\epsilon}(X,d) \geq\frac{c_1}{2} \epsilon^{1-n} \ .
\end{equation}
Let $Y \subset X$. Assume that $\cH^n(X) < \infty$ and that there are constants $C_2 > 0$ and $0 < \lambda_2 < \frac{1}{2}$ such that for all points $x,x' \in Y$ and all open sets $U \subset X$ with $\B(x,\lambda_2 r) \subset U$ and $\B(x,\lambda_2r) \subset X \setminus \bar U$ there holds
\[
\cH^{n-1}(\partial U) \geq C_2 r^{n-1} \ .
\]
Then, there is a constant $C > 0$ such that for all $0 < \epsilon < \diam(X)$ and all even numbers $k$,
\begin{equation}
\label{upperbound}
m_k(Y,d) \leq C \cH^n(X)^\frac{1}{n} k^\frac{n-1}{n} \ , \quad \text{ and } \quad m'_{\epsilon}(Y,d) \leq C \cH^n(X) \epsilon^{1-n} \ .
\end{equation}
\end{Prop}

\begin{proof}
If $X'_\epsilon \subset X$ is some $\epsilon$-separated subset of even cardinality realizing the inequality \eqref{firstinequality}, then obviously
\[
\frac{c_1}{2} \epsilon^{1-n} \leq \epsilon \frac{\# X'_\epsilon}{2} \leq m'_{\epsilon}(X'_\epsilon,d) \leq m'_\epsilon(X,d) \ .
\]
Assume the even number $k$ is big enough such that $\epsilon_k \defl C_1^{\frac{1}{n}}k^{-\frac{1}{n}} < \diam(X)$. If $X'_{\epsilon_k}$ is an $\epsilon_k$-separated set of even cardinality, then we have $\# X'_{\epsilon_k} \leq C_1 \epsilon_k^{-n} = k$ and hence
\[
\frac{c_1}{2} C_1^\frac{1-n}{n}k^\frac{n-1}{n} = \frac{c_1}{2} \epsilon_k^{1-n} \leq \epsilon_{k} \frac{\# X'_{\epsilon_k}}{2} \leq m(X'_{\epsilon_k},d) \leq m_k(X,d) \ .
\]
This holds for all but finitely many $k$, and \eqref{lowerbound} follows.

To prove the second statement let $X' \subset Y$ be a set of even cardinality. Let $f : X \to T$ be a $1$-Lipschitz map into a minimal metric tree $(T,d_T)$ as in Proposition~\ref{mainthmbis}. In particular for $\pi \in \cM(X',d)$,
\begin{equation}
\label{matchingidentities}
\cH^1(T) = m(X',d) = m(\pi,d) = m(\pi,f^*d_T) = m(X',f^*d_T) \ .
\end{equation}
By the coarea inequality, see e.g.\ \cite[Theorem~2.10.25]{F}, we then get
\begin{equation}
	\label{coareainequality}
	\frac{\alpha_{n-1}\alpha_1}{\alpha_n}\cH^n(X) \geq \int_T^* \cH^{n-1}(f^{-1}(q)) \, d\cH^{1}(q) \ .
\end{equation}
Because $X$ is compact, the map $q \mapsto \cH^{n-1}(f^{-1}(q))$ is measurable by the statement in \cite[Subsection~2.10.26]{F}. Hence the upper integral on the right-hand side above can be replaced by the usual Lebesgue integral. By Proposition~\ref{mainthmbis}, $T$ can be expressed as $\cup_{\{x,y\} \in \pi} [f(x),f(y)]$ and the pairwise overlaps of these intervals have $\cH^1$-measure zero. For $\{x,y\} \in \pi$ we define the set
\[
G(\{x,y\}) \defl \{q \in [f(x),f(y)]\ |\ d_T(f(x),q),d_T(f(y),q) \geq \lambda_2 d_T(f(x),f(y)) \} \ .
\]
For any $q \in G(\{x,y\})$ the set $f^{-1}(q)$ separates $x$ and $y$ in $X$ and both $d(x,f^{-1}(q))$ and $d(y,f^{-1}(q))$ are bounded from below by $\lambda_2 d(x,y)$ since $f$ is $1$-Lipschitz and $d_T(f(x),f(y)) = d(x,y)$ by \eqref{matchingidentities}. Hence by our assumptions on $X$ and \eqref{coareainequality},
\begin{align*}
\frac{\alpha_{n-1}\alpha_1}{\alpha_n}\cH^n(X) & \geq \sum_{\{x,y\} \in \pi} \int_{[f(x),f(y)]} \cH^{n-1}(f^{-1}(q)) \, d\cH^{1}(q) \\
	& \geq \sum_{\{x,y\} \in \pi} \int_{G(\{x,y\})} C_2 d(x,y)^{n-1} \, d\cH^{1}(q) \\
	& \geq \sum_{\{x,y\} \in \pi} (1-2\lambda_2) d(x,y) C_2 d(x,y)^{n-1} \\
	& = (1-2\lambda_2)C_2 \sum_{\{x,y\} \in \pi} d(x,y)^{n} \ .
\end{align*}
Therefore, $\sum_{\{x,y\} \in \pi} d(x,y)^{n} \leq C'\cH^n(X)$ for some constant $C'$ independent of $\pi$. If $\# X' = k$, then by the power mean inequality
\[
\sum_{\{x,y\} \in \pi} d(x,y)^{n} \geq (2^{-1}k)^{1-n}\Bigg( \sum_{\{x,y\} \in \pi} d(x,y) \Bigg)^n = (2^{-1}k)^{1-n}m(X',d)^n \ ,
\]
and hence $m(X',d) \leq 2C' \cH^n(X)^\frac{1}{n}k^\frac{n-1}{n}$. If $X$ is $\epsilon$-separated, then
\[
\epsilon^{n-1}m(X',d) \leq \sum_{\{x,y\} \in \pi} d(x,y)^{n} \leq C'\cH^n(X) \ .
\]
By taking the supremum over all such $X'$, the upper bound on $m_{k}(Y,d)$ and $m'_{\epsilon}(Y,d)$ follows.
\end{proof}

This can be applied to balls in an Ahlfors regular space that supports a Poincar\'e inequality. A \emph{metric measure space} $(X,d,\mu)$ is a metric space $(X,d)$ equipped with a Borel measure $\mu$. This space is \emph{Ahlfors regular of dimension $n$} with constants $0 < c_A \leq C_A$ if for all $x \in X$ and $r > 0$,
\begin{equation}\label{ahlfors}
c_A r^n \leq \mu(\B(x,r)) \leq C_A r^n \ .
\end{equation}
$(X,d,\mu)$ supports a \emph{weak Poincar\'e inequality} if there are constants $\lambda_P \geq 1$, $C_P > 0$ such that for all continuous functions $u : X \to \R$, their upper gradients $g$ and all balls $B = \B(x,r)$,
\begin{equation}\label{poincare}
\dashint_{B} |u - u_B| \, d\mu \leq C_P r \dashint_{\lambda_P B} g \, d\mu \ .
\end{equation}
Here, $\dashint_{B} = \frac{1}{\mu(B)}\int_{B}$ and $u_B = \dashint_{B}u$.

\begin{Cor}\label{maxmatch}
Let $(X,d,\mu)$ be a complete metric measure space that is Ahlfors regular of dimension $n > 1$ and supports a weak Poincar\'e inequality. Then there are constants $0 < c \leq C$, such that for all $x \in X$, $r > 0$, $k \in 2\N$ and $\epsilon < \diam(\B(x,r))$,
\begin{align*}
	c r k^\frac{n-1}{n} & \leq m_k(\B(x,r),d) \leq C r k^\frac{n-1}{n} \ , \\
	c r^n \epsilon^{1-n} & \leq m'_\epsilon(\B(x,r),d) \leq C r^n \epsilon^{1-n} \ .
\end{align*}
\end{Cor}

\begin{proof}
Fix some $x \in X$ and $r > 0$. The Ahlfors regularity implies that $\mu$ is doubling and comparable to the $n$-dimensional Hausdorff measure. Moreover, there is some constant $0 < c' \leq 2$ such that $\diam(\B(x,r)) \geq c' r$. It is rather direct to check that this implies the first assumption \eqref{firstinequality} of Proposition~\ref{axiomaticmatching}. Indeed, let $0 < r'\leq r$ and consider a maximal $r'$-separated set $X'$ in $\B(x,r)$. Then the balls $\oB(x',\frac{r'}{2})$ are pairwise disjoint subsets of $\B(x,2r)$ and hence
\[
(\# X') c\Bigl(\frac{r'}{2}\Bigr)^n \leq \mu\Bigl(\B\Bigl(X',\frac{r'}{2}\Bigr)\Bigr) \leq \mu(\B(x,2r)) \leq C(2r)^n \ .
\]
Moreover, because $X'$ is maximal, the set $\B(X',r')$ covers $\B(x,r)$ and hence
\[
(\# X')C(r')^n \geq \mu(\B(X',r')) \geq \mu(\B(x,r)) \geq c r^n \ .
\]
This shows that up to some constants independent of $x$ and $r$, $\# X'$ is comparable to $(\frac{r}{r'})^n$ which proves \eqref{firstinequality}. Moreover, since $X$ is complete and by the consideration above, balls in $X$ are totally bounded and hence compact, verifying the other condition required of the metric spaces in Proposition~\ref{axiomaticmatching}.

Because $X$ supports a weak Poincar\'e inequality, it follows from Theorem~5.1 and Theorem~10.3 in \cite{HK} that for all balls $B = \B(x,r) \subset X$, continuous functions $u$ and their upper gradients $g$,
\[
\left(\dashint_{B} |u - u_B|^\frac{n}{n-1} \, d\mu\right)^\frac{n-1}{n} \leq C_P' r \dashint_{\lambda_P' B} g \, d\mu \ .
\]
By \cite[Theorem~1.1]{KL}, this weak $(\frac{n}{n-1},1)$-Poincar\'e inequality implies that there exist constants $C_S > 0$ and $\lambda_S \geq 1$ such that for all balls $B$ and all Borel measurable $E \subset B$,
\begin{equation}\label{borelsetbound}
\left(\frac{\min\{\cH^n(B \cap E), \cH^n(B \setminus E)\}}{\cH^n(B)}\right)^\frac{n-1}{n} \leq C_S r \frac{\cH^{n-1}(\lambda_S B \cap \partial E)}{\cH^{n}(\lambda_S B)} \ .
\end{equation}
In order to apply the second part of Proposition~\ref{axiomaticmatching} fix some ball $B = \B(x,r)$ and let $x_1,x_2 \in B$. For $s < \frac{1}{2}d(x_1,x_2)$, the balls $\B(x_1,s)$ and $\B(x_2,s)$ are disjoint and contained in $2B$. If $U$ is some open set in $X$ with $\B(x_1,s) \subset \bar U$ and $\B(x_2,s) \subset X \setminus U$, then by \eqref{ahlfors} and \eqref{borelsetbound}, for some constant $c' > 0$,
\begin{align*}
c'\frac{s^{n-1}}{r^{n-1}} & \leq \left(\frac{\min\{\B(x_1,s), \B(x_2,s)\}}{\cH^n(2B)}\right)^\frac{n-1}{n} \\
 & \leq \left(\frac{\min\{\cH^n(2B \cap \bar U), \cH^n(2B \setminus U)\}}{\cH^n(2B)}\right)^\frac{n-1}{n} \\
 & \leq C_S 2r \frac{\cH^{n-1}(\lambda_S 2 B \cap \partial U)}{\cH^{n}(\lambda_S 2 B)} \ .
\end{align*}
If we set $s = \frac{1}{3}d(x_1,x_2)$, then $C' d(x_1,x_2)^{n-1} \leq \cH^{n-1}(\lambda_S 2 B \cap \partial U)$ for some constant $C' > 0$ independent of $x$ and $r$. Since $\cH^n(\lambda_S 2 B)$ is bounded by a fixed multiple of $r^n$, we get by Proposition~\ref{axiomaticmatching}, $m_k(B,d) \leq C r k^\frac{n-1}{n}$ and $m'_{\epsilon}(B,d) \leq C r^n \epsilon^{1-n}$.
\end{proof}

The assumptions of this Corollary are satisfied for example by Carnot groups equipped with the Carnot-Carath\'eodory metric with homogeneous dimension $n$, see e.g.\ \cite[Proposition~11.17]{HK} and the references there. Or more simply, they are satisfied for normed vector spaces of dimension $n$, and in this case there are also more elementary proofs of the second assumption in Proposition~\ref{axiomaticmatching} not relying on the Poincar\'e inequality.

From the statement of Proposition~\ref{axiomaticmatching} and its application to Corollary~\ref{maxmatch} we see that up to some multiplicative constant, the matching number $m_k$ for balls or cubes in $\R^n$ are realized by distributing the points as equally as possible and behaves like $k^\frac{n-1}{n}$. This motivates the definition of the \emph{matching dimension} of a bounded metric space $X$ as the number
\[
\dim_{\text{m}}(X) \defl \inf\Bigl\{n \in [1,\infty]\ |\ \exists \, C \geq 0 \text{ s.t.\ } \forall k \in 2\N, \, m_k(X) \leq C k^\frac{n-1}{n}\Bigr\} \ .
\]
To see some interesting behavior of this notion of dimension we consider examples of compact metric trees. As in Proposition~\ref{mainthmbis} it is rather direct to check that $m_k(T) \leq \cH^1(T)$ for all $k$ and any compact tree $T$. Hence if $\cH^1(T) < \infty$, then $\dim_{\text{m}}(T) = 1$. On the other side, for any decreasing sequence $\epsilon_1 \geq \epsilon_2 \geq \dots > 0$ with $\lim_{m \to \infty 0} \epsilon_m = 0$ we can construct the compact metric tree $T$ obtained by gluing the countable collection of disjoint copies of closed intervals $[0,\epsilon_m]$ by identifying the point $0$ in all the copies. By taking the $2k$ points corresponding to the value $\epsilon_m$ in the interval $[0,\epsilon_m]$ for $m = 1,\dots,2k$, we see that $m_{2k}(T) \geq \sum_{m=1}^{2k} \epsilon_m$. Since the maximum of $\cH^1(T')$ taken over all subtrees $T' \subset T$ spanned by $2k$ points in $T$ is also equal to this number we get that indeed,
\[
m_{2k}(T) = \sum_{m=1}^{2k} \epsilon_m \ .
\]
Hence, if the sequence $(\epsilon_m)$ can be chosen in such a way that for all $k$,
\begin{equation}
\label{sequenceassumption}
\sum_{m=1}^{2k} \epsilon_m = k^\frac{n-1}{n} \ ,
\end{equation}
then we obtain $\dim_{\text{m}}(T)= n$. But \eqref{sequenceassumption} is easy to achieve since the successive differences of the right-hand side satisfy,
\[
1 \geq k^\frac{n-1}{n} - (k-1)^\frac{n-1}{n} > (k+1)^\frac{n-1}{n} - k^\frac{n-1}{n} \xrightarrow{k \to \infty} 0 \ .
\]
There are also such trees with $\dim_{\text{m}}(T) = \infty$. Note that for this class of examples we have $\dim_{\text{H}}(T) = 1$ for the Hausdorff dimension and $\dim_{\text{A}}(T) = \infty$ for the Assouad dimension. This shows that ranging over all compact metric trees $T$ with $\dim_{\text{H}}(T) = 1$, the matching dimension $\dim_{\text{m}}(T)$ can realize any real number in $[1,\infty]$.

\subsection{Infinite matchings}

We now consider the case where $X$ could be infinite. The main difference with the finite case is that in this setting it is not true in general that a minimum matching exists, as shown by Example~\ref{nominmatch} below. Such pathological examples exist, even though there are less competitors for the minimization, already for the oriented case, i.e.\ for the optimal transportation problem for infinite sets of points, as explained in Remark~\ref{nominmatchz}. We fix now the most general notion of minimization for matchings for infinite $X$, which in the case of locally finite $X\subset \tilde X=\mathbb R$ with a special kind of distance was studied in \cite{dssobolevski}, \cite{mccann}:

\begin{Def}[matching, locally minimal matching, finite matching]\label{locminmatch}
Let $(X,d)$ be a possibly infinite pseudometric space, and consider a partition $\pi$ of $X$ into cardinality-$2$ sets. We say that $\pi$ is a matching for $X$ if for finite subsets of couples $A \subset \pi$ the sum of $d(x,y)$ for $\{x,y\} \in A$ is always finite. We further say that $\pi$ is a locally minimal matching for $(X,d)$ if for any other matching $\pi'$ of $X$ such that the symmetric difference $\pi\Delta\pi'$ is finite there holds
\[
 \sum_{\{x,y\}\in\pi\setminus \pi'} d(x,y) \le \sum_{\{x',y'\}\in\pi'\setminus \pi} d(x',y') \ .
\]
We say that a partition $\pi$ of $X$ is a finite matching in case $\sum\{d(x,y)\ |\ \{x,y\}\in\pi\} < \infty$. As in the finite case we denote this number by $m(\pi,d)$.
\end{Def}
 
In particular, if a finite matching exists then $X$ is countable. We then have the following result.

\begin{Prop}[duality for infinite matchings]\label{thmexistsinfinite}
Let $(X,d)$ be a countable metric space for which the completion $\bar X$ is compact and for which there exists a finite, locally minimal matching $\pi$. Then there exists a compact metric tree $T$ and a $1$-Lipschitz function $f : X\to T$ such that $m(\pi,f^*d_T) = m(\pi,d)$ and $\pi$ is locally minimal for $(X,f^*d_T)$ too.
\end{Prop}

\begin{Expl}[an $X$ with no minimal matching]\label{nominmatch}
Consider $X \defl \{0\}\cup\{2^{-i}\ |\ i\in\mathbb N\}\subset\mathbb R$. This set obviously has some finite matching and in any such matching $\pi$ the limit point $0$ has to be matched with some point $x > 0$. The interval $[0,x]$ then contains another point $x'$ that is paired with some $x'' > x'$. But replacing the matches $\{0,x\},\{x',x''\}$ in $\pi$ with $\{0,x'\},\{x,x''\}$ gives a new matching $\pi'$ with a smaller matching number. So there does not exist a locally minimal matching.
\end{Expl}

\begin{Rem}[similar result for transport problems]\label{nominmatchz}
 We may reach a similar pathological example in the case of the minimization \eqref{calibration} for infinite sets of points $\{x_i^+\}_{i\in\mathbb N}, \{x_i^-\}_{i\in\mathbb N}$ by considering the example where the $x_i^+$ and the $x_i^-$ are respectively the right and left extremes of the segments met during the limit construction of a Cantor set starting from the interval $[0,1]$. In the case of the standard Cantor set the $x_i^-$ are $0$ and those $3$-adic points in $]0,1]$ such that in their expansion in base $3$ the last nonzero digit is a $2$ and the $x_i^-$ are the $3$-adic points in $[0,1]$ for which the last nonzero digit in base $3$ is a $1$. Then a similar reasoning as in Example \ref{nominmatch} applies. This topic was studied in \cite{ponce}.
\end{Rem}
\begin{proof}[Proof of Proposition~\ref{thmexistsinfinite}:]
Let $\{\{x_{2i-1},x_{2i}\} \ |\ i = 1,2,\dots\}$ be an enumeration of the pairs in $\pi$. Let $X_k \defl \{x_1,\dots,x_{2k}\} \subset X$ and $\pi_k$ the restriction of $\pi$ to this finite set. For each $k$, $\pi_k$ is a minimal matching on $X_k$ and applying Proposition~\ref{mainthmbis}, there is a $1$-Lipschitz function $f_k : \bar X \to T_k$ onto a minimal metric tree $(T_k,d_k)$ with 
\begin{equation}
\label{matchingidentity}
m(\pi_k,f_k^*d_k) = m(X_k,f_k^*d_k) = m(X_k,d) = m(\pi_k,d) = \cH^1(T) \ .
\end{equation}
If we can show that the sequence of trees $(T_k)$ is uniformly bounded and uniformly compact, it follows by a result of Gromov \cite{G} that there is a compact set $Z \subset \ell^\infty(\N)$ and isometric embeddings $\iota_k : T_k \hookrightarrow Z$ such that some subsequence of $(\iota_k(T_k))$ converges with respect to the Hausdorff distance.

The minimal trees $T_k$ as obtained in Proposition~\ref{mainthmbis} are compact and moreover,
\begin{equation}
\label{hausdorffbound}
\diam(T_k) \leq \cH^1(T_k) = m(\pi_k,d) \leq m(\pi,d) < \infty \ .
\end{equation}
Hence the sequence $(T_k)$ is uniformly bounded. Let $S_k \subset T_k$ be a maximal $\epsilon$-separated set. If $\diam(T_k) < \frac{\epsilon}{2}$, then $\#S_k = 1$. Otherwise, for any $p \in S_k$, $\cH^1(\mathbf B(p,\frac{\epsilon}{2})) \geq \frac{\epsilon}{2}$ and hence
\[
\frac{\epsilon}{2}\#S_k \leq \cH^1(T_k) \ .
\]
Using \eqref{hausdorffbound} this implies that for every $\epsilon > 0$ there is a $N(\epsilon)$ such that every $T_k$ can be covered by $N(\epsilon)$ balls of radius $\epsilon$, i.e.\ the sequence $(T_k)$ is uniformly compact. As noted before, this implies the existence of a compact subspace $T \subset Z$ (with the induced metric $d_\infty$ of $\ell^\infty(\N)$) such that $\lim_{l \to \infty} d_{\text{H}}(\iota_{k_l}(T_{k_l}),T) = 0$ for some subsequence of $(T_k)$. As a limit of compact geodesic spaces, $T$ is itself geodesic, see e.g.\ \cite[Proposition~5.38]{BH}. Since all the $T_k$ satisfy the four-point condition \eqref{fourpointcond}, it is easy to check that $T$ does too and hence $T$ is a compact metric tree. Since $\bar X$ is compact and all the maps $\iota_{k} \circ f_{k}$ are $1$-Lipschitz with values in a common compact metric space $Z$, the Arzel\`a-Ascoli theorem guarantees a subsequence of $(\iota_{k_l} \circ f_{k_l})$ that converges uniformly to some $1$-Lipschitz function $f : \bar X \to Z$ (we will use the same indices for this subsequence). The image of $\iota_{k_l} \circ f_{k_l}$ is in $\iota_{k_l}(T_{k_l})$ and hence the image of $f$ is contained in $T$. Because of \eqref{hausdorffbound}, we have for any pair $\{x_{2i-1},x_{2i}\} \in \pi$ and all $k \geq i$,
\begin{equation*}
d(x_{2i-1},x_{2i}) = d_k(f_k(x_{2i-1}), f_k(x_{2i})) = d_\infty(\iota_k(f_k(x_{2i-1})),\iota_k(f_k(x_{2i}))) \ .
\end{equation*}
Hence, by taking the limit of the functions $f_k$ we get $d_\infty(f(x_{2i-1}), f(x_{2i}))=d(x_{2i-1},x_{2i})$ for all $i$. This in particular shows that $m(\pi, f^*d_\infty) = m(\pi,d)$. We also have that $\pi$ is locally minimal for $(X, f^*d_\infty)$. Indeed otherwise there would be a matching $\pi'$ of $X$ and some $j$ such that $\{x_{2i-1}, x_{2i}\} \in \pi'$ for all $i > j$ and $m(\pi', f^*d_\infty) < m(\pi,f^*d_\infty)$. If $\pi'_k$ denotes the restriction of $\pi'$ to $X_k$, this would give $m(\pi'_k, f_k^*d_k) < m(\pi_k, f_k^*d_k)$ if $k$ is big enough, contradicting \eqref{matchingidentity}.
\end{proof}


\end{document}